\documentclass[a4paper,11pt]{article}

\usepackage{authblk}
\usepackage{parskip}
\usepackage{mathtools}
\usepackage{amssymb}
\usepackage{amsmath}
\usepackage{latexsym}
\usepackage{mathrsfs}  % example: \mathscr{A} 
\usepackage{bbm}       % example: \mathbbm{N}
\usepackage{amsthm}    % Theorems
\usepackage{relsize}   % Change default sizes 
\usepackage{graphicx}
\usepackage{thmtools}  % Modify theorem style
\usepackage{mathrsfs}  % Ralph Smith’s Formal Script Font (rsfs)
\usepackage{paralist}  % list on paragraph
\usepackage{lipsum}    % drawing horizontal lines
\usepackage[all]{xy}   % To draw diagrams 

\numberwithin{equation}{section}

\usepackage{titlesec}   %Control what is at the format of section titles

\titleformat{\section}[block]{\bfseries\filcenter}% centred title 
{{\upshape\thesection\enspace}}{.5em}{}

\titleformat{\subsection}[block]{\filcenter}% upshape means that is not in italics
{{\upshape\thesubsection\enspace}}{.5em}{} %%^: These reduce the amount of space taken up by section and subsection

\usepackage{enumitem}  % Modify list
\setlist{nosep}  % nosep= no vertical separation between items.
\setitemize[0]{leftmargin=*}
\setenumerate[0]{leftmargin=*}
\setenumerate[1]{label={(\arabic*)}} % Global setting of enumeration of the form (1), (2), ..
\newcommand{\N}{\mathbb{N}}     % Natural numbers
\newcommand{\R}{\mathbb{R}}     % Real numbers
\newcommand{\C}{\mathbb{C}}     % Complex numbers
\newcommand{\Prob}{\mathbb{P}}  % Probability measure
\newcommand{\Exp}{\mathbb{E}}   % Expectation 
\newcommand{\st}{\,:\,}         % Symbol for "such that"
\newcommand{\goth}[1]{\mathfrak{#1}} % Gothic letters 
\newcommand{\ind}[2]{\mathbbm{1}_{#1}\left( #2 \right)}          % Indicator function, #1=set, #2=variable
\newcommand{\norm}[1]{\left|\left|#1\right|\right|}              % Vector norm
\newcommand{\triplet}[3]{\left( #1, #2, #3 \right) }             % General triplet e.g. a probability space
\newcommand{\ProbSpace}{\triplet{\Omega}{\mathscr{F}}{\Prob}}    % Triplet of a Probability Space
\newcommand{\abs}[1]{\left| #1 \right|}                          % Absolute value  
\renewcommand{\qedsymbol}{$\square$}                       % For a black square at the end of a proof

\newcommand{\defeq}{\mathrel{\mathop:}=}                         % Defined as equal to symbol
\newcommand\restr[2]{{% we make the whole thing an ordinary symbol
  \left.\kern-\nulldelimiterspace % automatically resize the bar with \right
  #1 % the function
  \vphantom{\big|} % pretend it's a little taller at normal size
  \right|_{#2} % this is the delimiter
  }}
  
\theoremstyle{plain} 
\newtheorem{theo}{Theorem}[section]    
\newtheorem{prop}[theo]{Proposition} 
\newtheorem{coro}[theo]{Corollary}
\newtheorem{lemm}[theo]{Lemma}
\newtheorem{assu}[theo]{Assumption}

\theoremstyle{definition} 
\newtheorem{defi}[theo]{Definition}

\newtheorem{note}[theo]{Note}
\newtheorem{rema}[theo]{Remark}

\declaretheoremstyle[%
  spaceabove=-5pt,%
  spacebelow=6pt,%
  headfont=\normalfont\itshape,%
  postheadspace=1em,%
  qed=\qedsymbol%
]{mystyle} 
\declaretheorem[name={Proof},style=mystyle,unnumbered,
]{prf}

\begin{document}

 \title{Tightness and Weak Convergence of Probabilities on the Skorokhod Space on the Dual of a Nuclear Space and Applications.}
 
\author{C. A. Fonseca-Mora\\
Escuela de Matem\'{a}tica\\
Universidad de Costa Rica\\
San Jos\'{e}, 11501-2060, Costa Rica\\
E-mail:  christianandres.fonseca@ucr.ac.cr }

\date{}

\maketitle


\begin{abstract}
Let $\Phi'_{\beta}$ denote the strong dual of a nuclear space $\Phi$ and let  $D_{T}(\Phi'_{\beta})$ be the Skorokhod space of right-continuous with left limits (c\`{a}dl\`{a}g) functions from $[0,T]$ into $\Phi'_{\beta}$. In this article we  introduce the concepts of cylindrical random variables and cylindrical measures on $D_{T}(\Phi'_{\beta})$, and prove analogues of the regularization theorem and Minlos theorem for extensions of these objects to bona fide random variables and probability measures on $D_{T}(\Phi'_{\beta})$. Later, we establish analogues of L\'{e}vy's continuity theorem to provide necessary and sufficient conditions for tightness of a family of probability measures on $D_{T}(\Phi'_{\beta})$ and sufficient conditions for weak convergence of a sequence of probability measures on $D_{T}(\Phi'_{\beta})$. Extensions of the above results to the space $D_{\infty}(\Phi'_{\beta})$ of c\`{a}dl\`{a}g functions from $[0,\infty)$ into $\Phi'_{\beta}$  are also given.  Afterwards, we apply our results to  the study of weak convergence of $\Phi'_{\beta}$-valued c\`{a}dl\`{a}g processes and in particular to L\'{e}vy processes. We finalize with an application of our theory to the study of tightness and weak convergence of probability measures on the Skorokhod space $D_{\infty}(H)$ where $H$ is a Hilbert space.
\end{abstract}

\section{Introduction}

Let $E$ be a topological space and let  $D_{T}(E)$ denote the collection of all right-continuous with left limits (c\`{a}dl\`{a}g) maps $x: [0,T] \rightarrow E$. For the case of $E$ being a separable metric space, Skorokhod introduced in \cite{Skorokhod:1956} four topologies on the space $D_{T}(E)$, being the $J1$ topology the most widely used. 

Under the assumption that $\Phi$ is a Fr\'{e}chet nuclear space with strong dual $\Phi'_{\beta}$, Mitoma \cite{Mitoma:1983} introduced the Skorokhod $J1$ topology on $D_{T}(\Phi'_{\beta})$ and provided characterizations for compact subsets on it. Mitoma also introduced sufficient conditions for uniform tightness and weak convergence of sequences of probability measures on $D_{T}(\Phi'_{\beta})$ in terms of uniform tightness and weak convergence of their finite dimensional projections. The work of Mitoma was latter extended by Fouque \cite{Fouque:1984} to the cases when $\Phi$ is either a countable inductive limit of Fr\'{e}chet nuclear spaces or the strong dual of a Fr\'{e}chet nuclear space. 

A further extension of the work of Mitoma to $D_{T}(E)$, where $E$ is a completely regular space was carried out by Jakubowski \cite{Jakubowski:1986}. In this work, Jakubowski assume that $E$ has metrizable compacts and that  $\{ \mu_{i}\}$ is a family of probability measures on $D_{T}(E)$ satisfying the compact containtment condition and  such that $\mu_{i} \circ f^{-1}$ is uniformly tight on $D_{1}(\R)$ for a set $\mathbbm{F}$ of continuous functions $f: D_{1}(E) \rightarrow D_{1}(\R)$ that satisfy certain conditions. In a recent work and under the same assumptions on $E$, Kouritzin \cite{Kouritzin:2016} introduces  characterizations of uniform tightness under the compact containment condition and under (several equivalent) modulus of continuity conditions. 

The main objective of this article is to provide sufficient and necessary conditions for tightness and weak convergence of random objects on $D_{T}(\Phi'_{\beta})$, where  $\Phi'_{\beta}$ is the strong dual of a general nuclear space $\Phi$, or more generally when $\Phi$ is a Hausdorff locally convex space. We do this by studying properties of the Fourier transforms of these random objects and by proving analogues of Minlos theorem and L\'{e}vy's continuity theorem on $D_{T}(\Phi'_{\beta})$. 

Our motivation is twofold. First, since the pionering work of Mitoma many applications emerged, as are for example  \cite{DawsonVaillancourtWang:2000, DeMasiGalvesLocherbachPresutti:2015, FernandezGorostiza:1992, HaademProske:2014, KaspiRamanan:2013, PerezAbreuTudor:1992, ReedTalreja:2015}, just to cite some of them. Since we are considering general nuclear spaces, we hope that with our work more applications will appear, especially for modelling of random phenomena taking values on other examples of nuclear spaces not covered by the works of Mitoma and Fouque (see Sect.  \ref{sectionExampCommen}). Second, in \cite{FonsecaMora:2018-1} a new theory of stochastic integration and stochastic PDE's in $\Phi'_{\beta}$ driven by L\'{e}vy noise has been introduced. Much of the work on this article is motivated to show convergence of solutions of these stochastic PDE's. The results will appear elsewhere.  

We now give a description of our work. Our first task is to characterize the compact subsets of $D_{T}(\Phi'_{\beta})$. We show that under the assumption that $\Phi$ is a barrelled nuclear space, then for a set $A \subseteq D_{T}(\Phi'_{\beta})$  compactness of finite dimensional projections of $A$ implies compactness of $A$ in $D_{T}(\Phi'_{\beta})$ (Theorem \ref{theoCharacCompacSets}). This extends previous results obtained by Mitoma \cite{Mitoma:1983}. 

Later, we introduce the concepts of cylindrical measures and cylindrical random variables on $D_{T}(\Phi'_{\beta})$ by considering the space-time algebra of cylindrical subsets of $D_{T}(\Phi'_{\beta})$. Here it is important to stress the fact that our definitions are not a particular case of the usual theory of cylindrical measures and cylindrical random variables on locally convex spaces as it is well-known that  $D_{T}(\Phi'_{\beta})$ is not a topological vector space. Here, we show an extended version of the regularization theorem given in \cite{FonsecaMora:2018} that says that if $\{ X_{t} \}_{t \in [0,T]}$ is a cylindrical process in $\Phi'$ wherein  the maps $X_{t}: \Phi \rightarrow L^{0} \ProbSpace$ are equicontinuous at the origin, then this cylindrical process has an extension to a $D_{T}(\Phi'_{\beta})$-valued random variable with a Radon probability distribution (Theorem \ref{theoRegulaTheoSkorokSpace}). We also show an extension of Minlos theorem (Theorem \ref{theoMinlosSkorokhodSpace}) that states that a cylindrical measure on $D_{T}(\Phi'_{\beta})$ that has equicontinuous Fourier transforms for its time projections has a Radon measure extension on $D_{T}(\Phi'_{\beta})$. 

Afterwards, we move to the core of this article that consists in establishing necessary and sufficient conditions for a family of probability measures $\{ \mu_{\alpha}: \alpha \in A\}$ on $D_{T}(\Phi'_{\beta})$ to be uniformly tight (Theorem \ref{theoThighnessMeasures}). In particular, we show that if the finite dimensional projections of the measures are uniformly tight and if the Fourier transforms of the time projections of the measures are equicontinuous at the origin, then $\{ \mu_{\alpha}: \alpha \in A\}$ is uniformly tight on $D_{T}(\Phi'_{\beta})$. Observe that contrary to  
\cite{Jakubowski:1986, Kouritzin:2016} we do not assume that the compact subsets of $\Phi'_{\beta}$ are metrizable nor that the compact containtment conditions holds. Furthermore, we show that if the space $\Phi$ is also ultrabornological then only the uniform tightness of finite dimensional projections needs to be assumed. We extend our results to the space 
$D_{\infty}(\Phi'_{\beta})$ of c\`{a}dl\`{a}g mappings from $[0,\infty)$ into $\Phi'_{\beta}$. Our results generalize those obtained by  Mitoma \cite{Mitoma:1983} and Fouque \cite{Fouque:1984}. 

At the center of our arguments is the idea of using equicontinuity of Fourier transforms of the time projections of the measures $\{\mu_{\alpha}: \alpha \in A\}$ on $D_{T}(\Phi'_{\beta})$ to set the problem on the space $D_{T}( (\widetilde{\Phi_{\theta}})'_{\beta})$ equipped with its Skorokhod topology, where 
$\widetilde{\Phi_{\theta}}$ denote the completion of the space $\Phi$ equipped with a weaker (with respect to the nuclear topology on $\Phi$) countably Hilbertian topology $\theta$. The advantage of using this methodology is that the space $\widetilde{\Phi_{\theta}}$ is a complete, separable, pseudo-metrizable space (not necessarily nuclear), hence linear operators and measures defined on $\widetilde{\Phi_{\theta}}$ have better properties than on $\Phi$. 
Previously, we have used this tool in \cite{FonsecaMora:2018} to prove existence of continuous or c\`{a}dl\`{a}g versions to cylindrical processes in $\Phi'$. 

Our next goal is to provide sufficient conditions for weak convergence of probability measures on $D_{\infty}(\Phi'_{\beta})$ and of  $\Phi'_{\beta}$-valued c\`{a}dl\`{a}g processes (Theorems \ref{theoWeakConveMeasures} and \ref{theoWeakConvProcesses}). Again our results generalize those obtained in \cite{Fouque:1984, Mitoma:1983}, and furthermore we have considered the completely new case of convergence of cylindrical processes in $\Phi'$. Applications are then given to weak convergence in $D_{\infty}(\Phi'_{\beta})$ for a sequence of $\Phi'_{\beta}$-valued L\'{e}vy processes in terms of properties of the characteristics of their L\'{e}vy-Khintchine formula (Theorem \ref{theoWeakConvLevyProcess}). 

Finally, under the assumption that $\Phi$ is a (Hausdorff) locally convex space and by considering its Sazonov topology, we indicate how our methods for the nuclear space setting extends to provide sufficient conditions for uniform tightness and weak convergence of probability measures on $D_{\infty}(\Phi'_{\beta})$ (Theorems \ref{theoUniformTightnessInftylocalConvex} and \ref{theoWeakConveMeasureslocallyConvex}). A particular case of great importance is when $H$ is a Hilbert space because in that case our result represents an extension of Sazonov's theorem and L\'{e}vy's continuity theorem  to the space $D_{\infty}(H)$ (Theorems \ref{theoSazonovSkorokhodHilbertSpace} and \ref{theoLevyTheoremSkorokhodHilbertSpace}). We hope that these results could generate new applications, especially for the Hilbert space setting. 

The organization of the paper is the following. In Sect. \ref{sectionPrelim} we list some important notions on nuclear spaces and their duals, and also properties of cylindrical measures and cylindrical processes in duals of nuclear spaces. 
The Skorokhod topology on $D_{T}(\Phi'_{\beta})$ is introduced in Sect. \ref{sectionSkoSpac} and  characterizations for its compact subsets are given.
In Sect. \ref{sectionMeRVSkoSpac} we introduce the concepts of cylindrical measures and cylindrical random variables in $D_{T}(\Phi'_{\beta})$ and show the regularization and Minlos theorems in $D_{T}(\Phi'_{\beta})$. Later, in Sect. \ref{sectionTightSkorSpa} we study the uniform tightness of probability measures on $D_{T}(\Phi'_{\beta})$ and on $D_{\infty}(\Phi'_{\beta})$. In Sect. \ref{sectionWeaConvSkoSpa} we prove a L\'{e}vy's continuity theorem for the weak convergence of probability measures and stochastic processes in $D_{\infty}(\Phi'_{\beta})$. Afterwards, in Sect. \ref{sectionWeakConvLevy} we apply our results to characterize weak convergence in $D_{\infty}(\Phi'_{\beta})$ of a sequence of L\'{e}vy processes. In  Sect. \ref{sectionApplSkoLocalConv} we show how our results for the dual of a nuclear space setting extends to the case when $\Phi$ is a locally convex space. Finally, in Sect.  \ref{sectionExampCommen} we consider concrete examples of nuclear spaces,  give some remarks, and compare our results with those on the literature.

%A criteria for weak convergence of a tight sequence of $\mathscr{S}'$-valued c\`{a}dl\`{a}g processes not in terms of weak convergence of finite dimensional distributions \cite{BojdeckiGorostizaRamaswamy:1986}.

\section{Preliminaries}\label{sectionPrelim}

\subsection{Nuclear Spaces And Their Strong Duals} \label{subsectionNuclSpace}

In this section we introduce our notation and review some of the key concepts on nuclear spaces and their dual spaces that we will need throughout this paper. For more information see \cite{Schaefer, Treves}. Only vector spaces over $\R$ will be considered.  

A locally convex space is called \emph{quasi-complete} if each of its bounded and closed subsets are complete. A \emph{barrelled} space is a locally convex space for which  every lower semicontinuous seminorm on it is continuous. A locally convex space that is the inductive limit of a family of normed (respectively Banach) spaces is called a \emph{bornological} (respectively \emph{ultrabornological}) space. 

Let $\Phi$ be a locally convex space. If $p$ is a continuous seminorm on $\Phi$ and $r>0$, the closed ball of radius $r$ of $p$ given by $B_{p}(r) = \left\{ \phi \in \Phi: p(\phi) \leq r \right\}$ is a closed, convex, balanced neighborhood of zero in $\Phi$. A continuous seminorm (respectively a norm) $p$ on $\Phi$ is called \emph{Hilbertian} if $p(\phi)^{2}=Q(\phi,\phi)$, for all $\phi \in \Phi$, where $Q$ is a symmetric, non-negative bilinear form (respectively inner product) on $\Phi \times \Phi$. Let $\Phi_{p}$ be the Hilbert space that corresponds to the completion of the pre-Hilbert space $(\Phi / \mbox{ker}(p), \tilde{p})$, where $\tilde{p}(\phi+\mbox{ker}(p))=p(\phi)$ for each $\phi \in \Phi$. The quotient map $\Phi \rightarrow \Phi / \mbox{ker}(p)$ has an unique continuous linear extension $i_{p}:\Phi \rightarrow \Phi_{p}$.   

Let $q$ be another continuous Hilbertian seminorm on $\Phi$ for which $p \leq q$. In this case, $\mbox{ker}(q) \subseteq \mbox{ker}(p)$. Moreover, the inclusion map from $\Phi / \mbox{ker}(q)$ into $\Phi / \mbox{ker}(p)$ is linear and continuous, and therefore it has a unique continuous extension $i_{p,q}:\Phi_{q} \rightarrow \Phi_{p}$. Furthermore, we have the following relation: $i_{p}=i_{p,q} \circ i_{q}$. 

We denote by $\Phi'$ the topological dual of $\Phi$ and by $f[\phi]$ the canonical pairing of elements $f \in \Phi'$, $\phi \in \Phi$. We denote by $\Phi'_{\beta}$ the dual space $\Phi'$ equipped with its \emph{strong topology} $\beta$, i.e. $\beta$ is the topology on $\Phi'$ generated by the family of seminorms $\{ \eta_{B} \}$, where for each $B \subseteq \Phi$ bounded we have $\eta_{B}(f)=\sup \{ \abs{f[\phi]}: \phi \in B \}$ for all $f \in \Phi'$.  If $p$ is a continuous Hilbertian seminorm on $\Phi$, then we denote by $\Phi'_{p}$ the Hilbert space dual to $\Phi_{p}$. The dual norm $p'$ on $\Phi'_{p}$ is given by $p'(f)=\sup \{ \abs{f[\phi]}:  \phi \in B_{p}(1) \}$ for all $ f \in \Phi'_{p}$. Moreover, the dual operator $i_{p}'$ corresponds to the canonical inclusion from $\Phi'_{p}$ into $\Phi'_{\beta}$ and it is linear and continuous. 

Let $p$ and $q$ be continuous Hilbertian seminorms on $\Phi$ such that $p \leq q$.
The space of continuous linear operators (respectively Hilbert-Schmidt operators) from $\Phi_{q}$ into $\Phi_{p}$ is denoted by $\mathcal{L}(\Phi_{q},\Phi_{p})$ (respectively $\mathcal{L}_{2}(\Phi_{q},\Phi_{p})$) and the operator norm (respectively Hilbert-Schmidt norm) is denoted by $\norm{\cdot}_{\mathcal{L}(\Phi_{q},\Phi_{p})}$ (respectively $\norm{\cdot}_{\mathcal{L}_{2}(\Phi_{q},\Phi_{p})}$). We employ an analogous notation for operators between the dual spaces $\Phi'_{p}$ and $\Phi'_{q}$. 
  
Let us recall that a (Hausdorff) locally convex space $(\Phi,\mathcal{T})$ is called \emph{nuclear} if its topology $\mathcal{T}$ is generated by a family $P$ of Hilbertian seminorms such that for each $p \in P$ there exists $q \in P$, satisfying $p \leq q$ and the canonical inclusion $i_{p,q}: \Phi_{q} \rightarrow \Phi_{p}$ is Hilbert-Schmidt. Other equivalent definitions of nuclear spaces can be found in \cite{Pietsch, Treves}. 

Let $\Phi$ be a nuclear space. If $p$ is a continuous Hilbertian seminorm  on $\Phi$, then the Hilbert space $\Phi_{p}$ is separable (see \cite{Pietsch}, Proposition 4.4.9 and Theorem 4.4.10, p.82). Now, let $\{ p_{n} \}_{n \in \N}$ be an increasing sequence of continuous Hilbertian seminorms on $\Phi$. We denote by $\theta$ the locally convex topology on $\Phi$ generated by the family $\{ p_{n} \}_{n \in \N}$. The topology $\theta$ is weaker than the nuclear topology on $\Phi$. We  will call $\theta$ a \emph{weaker countably Hilbertian topology} on $\Phi$ and we denote by $\Phi_{\theta}$ the space $(\Phi,\theta)$ and by $\widetilde{\Phi_{\theta}}$ its completion. The space $\widetilde{\Phi_{\theta}}$ is a separable, complete, pseudo-metrizable (hence Baire) locally convex space (see \cite{FonsecaMora:2018}, Proposition 2.4). Moreover, the space $\widetilde{\Phi_{\theta}}$ is ultrabornological because is bornological and complete (see Example 13.2.8(b) and Theorem 13.2.12 in \cite{NariciBeckenstein}, p.445, 449).

\subsection{Cylindrical and Stochastic Processes} \label{subSectionCylAndStocProcess}

Let $E$ be a topological space and denote by $\mathcal{B}(E)$ its Borel $\sigma$-algebra. Recall that a Borel measure $\mu$ on $E$ is called a \emph{Radon measure} if for every $\Gamma \in \mathcal{B}(E)$ and $\epsilon >0$, there exist a compact set $K \subseteq \Gamma$ such that $\mu(\Gamma \backslash K) < \epsilon$. In general not every Borel measure on $E$ is Radon.
We denote by $\goth{M}_{R}^{b}(E)$ and by $\goth{M}_{R}^{1}(E)$ the spaces of all bounded Radon measures and of all Radon probability measures on $E$.  A subset $M \subseteq \goth{M}_{R}^{b}(E)$ is called \emph{uniformly tight} if \begin{inparaenum}[(i)] 
\item $\sup \{ \mu(E) \st \mu \in M \} < \infty$, and \item for every $\epsilon >0$ there exists a compact set $K \subseteq E$ such that $\mu (K^{c})< \epsilon$ for all $\mu \in M$. 
\end{inparaenum} A sequence $(\mu_{n}:n \in \N) \subseteq \goth{M}_{R}^{1}(E)$ \emph{converges weakly} to $\mu \in \goth{M}_{R}^{1}(E)$ if $\int_{E} f d\mu_{n} \rightarrow \int_{E} f d\mu$ for every $f \in C_{b}(E)$; we write $\mu_{n} \Rightarrow \mu$.  

Let $\Phi$ be locally convex. Given $M \subseteq \Phi$, the cylindrical algebra on $\Phi'$ based on $M$ is the collection $\mathcal{Z}(\Phi',M)$ of all the \emph{cylindrical sets} of the form $\mathcal{Z}\left(\phi_{1}, \dots, \phi_{n}; A \right) = \left\{ f \in \Phi'\st \left(f[\phi_{1}], \dots, f[\phi_{n}]\right) \in A \right\}$
where $n \in \N$, $\phi_{1}, \dots, \phi_{n} \in M$ and $A \in \mathcal{B}\left(\R^{n}\right)$. 
The $\sigma$-algebra generated by $\mathcal{Z}(\Phi',M)$ is denoted by $\mathcal{C}(\Phi',M)$. If  $M$ is finite we have $\mathcal{C}(\Phi',M)=\mathcal{Z}(\Phi',M)$. Moreover, we always have $\mathcal{C}(\Phi')\defeq \mathcal{C}(\Phi',\Phi) \subseteq \mathcal{B}(\Phi'_{\beta})$, but equality is not true in general.
A function $\mu: \mathcal{Z}(\Phi',\Phi) \rightarrow [0,\infty]$ is called a \emph{cylindrical measure} on $\Phi'$ if for each finite subset $M \subseteq \Phi$ the restriction of $\mu$ to $\mathcal{C}(\Phi',M)$ is a measure. A cylindrical measure $\mu$ is said to be \emph{finite} if $\mu(\Phi')< \infty$ and a \emph{cylindrical probability measure} if $\mu(\Phi')=1$. The \emph{Fourier transform}\index{characteristic function} of $\mu$ is the function $\widehat{\mu}: \Phi \rightarrow \C$ defined by 
$$ \widehat{\mu}(\phi)= \int_{\Phi'} e^{i f[\phi]} \mu(df), \quad \forall \, \phi \in \Phi. $$

Let $\ProbSpace$ be a (complete) probability space. Denote by $L^{0} \ProbSpace$ the space of equivalence classes of real-valued random variables  defined on $\ProbSpace$.  We always consider the space $L^{0} \ProbSpace$ equipped with the topology of convergence in probability and in this case it is a complete, metrizable, topological vector space. 

A cylindrical random variable in $\Phi'$ is a linear map $X: \Phi \rightarrow L^{0} \ProbSpace$. 
If $Z=\mathcal{Z}\left(\phi_{1}, \dots, \phi_{n}; A \right)$ is a cylindrical set, for $\phi_{1}, \dots, \phi_{n} \in \Phi$ and $A \in \mathcal{B}\left(\R^{n}\right)$, let 
\begin{equation*} 
\mu_{X}(Z) \defeq \Prob \left( ( X(\phi_{1}), \dots, X(\phi_{n})) \in A  \right).
\end{equation*}
The map $\mu_{X}$ is a cylindrical probability measure on $\Phi'$ and it is called the \emph{cylindrical distribution} of $X$. The \emph{Fourier transform} of $X$ is defined to be the Fourier transform $\widehat{\mu}_{X}: \Phi \rightarrow \C$ of its cylindrical distribution $\mu_{X}$. 

Let $X$ be a $\Phi'_{\beta}$-valued random variable, i.e. $X:\Omega \rightarrow \Phi'_{\beta}$ is a $\mathscr{F}/\mathcal{B}(\Phi'_{\beta})$-measurable map. We denote by $\mu_{X}$ the probability distribution of $X$, i.e. $\mu_{X}(\Gamma)=\Prob \left( X \in  \Gamma \right)$, $\forall \, \Gamma \in \mathcal{B}(\Phi'_{\beta})$; it is a Borel probability measure on $\Phi'_{\beta}$. For each $\phi \in \Phi$ we denote by $X[\phi]$ the real-valued random variable defined by $X[\phi](\omega) \defeq X(\omega)[\phi]$, for all $\omega \in \Omega$. It is clear that the mapping $\phi \mapsto X[\phi]$ defines a cylindrical random variable. 

If $X$ is a cylindrical random variable in $\Phi'$, a $\Phi'_{\beta}$-valued random variable $Y$ is a called a \emph{version} of $X$ if for every $\phi \in \Phi$, $X(\phi)=Y[\phi]$ $\Prob$-a.e. A $\Phi'_{\beta}$-valued random variable $X$ is called \emph{regular} if there exists a weaker countably Hilbertian topology $\theta$ on $\Phi$ such that $\Prob( X \in (\Phi_{\theta})')=1$. 

Let $J=[0,\infty)$ or $J=[0,T]$ for some $T>0$. We say that $X=\{ X_{t} \}_{t \in J}$ is a \emph{cylindrical process} in $\Phi'$ if $X_{t}$ is a cylindrical random variable, for each $t \in J$. Clearly, any $\Phi'_{\beta}$-valued stochastic processes $X=\{ X_{t} \}_{t \in J}$ defines a cylindrical process under the prescription: $X[\phi]=\{ X_{t}[\phi] \}_{t \in J}$, for each $\phi \in \Phi$. We will say that it is the \emph{cylindrical process determined/induced} by $X$.
A $\Phi'_{\beta}$-valued processes $Y=\{Y_{t}\}_{t \in J}$ is said to be a $\Phi'_{\beta}$-valued \emph{version} of the cylindrical process $X=\{X_{t}\}_{t \in J}$ on $\Phi'$ if for each $t \in J$, $Y_{t}$ is a $\Phi'_{\beta}$-valued version of $X_{t}$.  

Let $X=\{ X_{t} \}_{t \in J}$ be a $\Phi'_{\beta}$-valued process.
We say that $X$ is \emph{continuous} (respectively \emph{c\`{a}dl\`{a}g}) if for $\Prob$-a.e. $\omega \in \Omega$, the \emph{sample paths} $t \mapsto X_{t}(w) \in \Phi'_{\beta}$ of $X$ are continuous (respectively right-continuous with left limits). We say that the process $X$ is \emph{regular} if for every $t \in J$, $X_{t}$ is a regular random variable. 

A result of fundamental importance in this work is the following: 

\begin{theo}[Regularization Theorem; \cite{FonsecaMora:2018}, Theorem 3.2]\label{theoRegularizationTheoremCadlagContinuousVersion}
Let $\Phi$ be a nuclear space. Let $X=\{X_{t} \}_{t \geq 0}$ be a cylindrical process in $\Phi'$ satisfying:
\begin{enumerate}
\item For each $\phi \in \Phi$, the real-valued process $X(\phi)=\{ X_{t}(\phi) \}_{t \geq 0}$ has a continuous (respectively c\`{a}dl\`{a}g) version.
\item For every $T > 0$, the family $\{ X_{t}: t \in [0,T] \}$ of linear maps from $\Phi$ into $L^{0} \ProbSpace$ is equicontinuous.  
\end{enumerate}
Then there exist a weaker countably Hilbertian topology $\theta$ on $\Phi$ 
and a $(\widetilde{\Phi_{\theta}})'_{\beta}$-valued continuous (respectively c\`{a}dl\`{a}g) process $Y= \{ Y_{t} \}_{t \geq 0}$, such that for every $\phi \in \Phi$, $Y[\phi]= \{ Y_{t}[\phi] \}_{t \geq 0}$ is a version of $X(\phi)= \{ X_{t}(\phi) \}_{t \geq 0}$. Moreover $Y$ is a $\Phi'_{\beta}$-valued, regular, continuous (respectively c\`{a}dl\`{a}g) version of $X$ that is unique up to indistinguishable versions. 
\end{theo}

\section{The Skorokhod topology in $D_{T}(\Phi'_{\beta})$} \label{sectionSkoSpac}

Let $\Phi$ be a (Hausdorff) locally convex space and let $\{ q_{\gamma}(\cdot): \gamma \in \Gamma \}$ be a family of seminorms generating the strong topology $\beta$ on $\Phi'$. Fix $T>0$ and denote by $D_{T}(\Phi'_{\beta})$ the collection of all c\`{a}dl\`{a}g (i.e. right-continuous with left limits) maps from $[0,T]$ into $\Phi'_{\beta}$. 

Following \cite{Jakubowski:1986} (see also \cite{Mitoma:1983}), for a given $ \gamma \in \Gamma$ we consider the pseudometric $d_{\gamma}$ on $D_{T}(\Phi'_{\beta})$ given by  
\begin{equation}\label{defSkorokhodPseudometrics}
d_{\gamma}(x,y)=\inf_{\lambda \in \Lambda}  \left\{ \sup_{t \in [0,T]} q_{\gamma}(x(t)-y(\lambda(t))) + \sup_{0 \leq s< t \leq T} \abs{\log \frac{\lambda(t)-\lambda(s)}{t-s}} \right\},
\end{equation} 
for all $x, y \in D_{T}(\Phi'_{\beta})$, where $\Lambda$ denotes the set of all the strictly increasing continuous maps $\lambda$ from $[0,T]$ onto itself. % with $\sup_{0 \leq s< t \leq T} \abs{\log \frac{\lambda(t)-\lambda(s)}{t-s}} < \infty$.  

The family of seminorms $\{d_{\gamma}: \gamma \in \Gamma\}$ generates a completely regular topology on $D_{T}(\Phi'_{\beta})$ that is known as the \emph{Skorokhod  topology} (also known as the $J1$ topology). This topology does not depend on the particular choice of seminorms $\{ q_{\gamma}(\cdot): \gamma \in \Gamma \}$ on $\Phi'_{\beta}$ (see \cite{Jakubowski:1986}, Theorem 1.3). 

%We can also consider the space $D_{[a,b]}(\Phi'_{\beta})$ of $\Phi'_{\beta}$-valued c\`{a}dl\`{a}g functions defined on $[a,b]$ for $-\infty<a<b<\infty$. The family of pseudometrics generating the Skorokhod topology in $D_{[a,b]}(\Phi'_{\beta})$ is defined as above with the obvious modifications. The completely regular spaces $D_{T}(\Phi'_{\beta})$ and $D_{[a,b]}(\Phi'_{\beta})$ are isometrically isomorphic (see \cite{Kouritzin:2016}, Lemma 9).

Let $\Phi$ be a nuclear space and let $q$ be a continuous seminorm on $\Phi$. Very important for our forthcoming developments is the space $D_{T}(\Phi'_{q})$. Observe that because $\Phi'_{q}$ is a separable Banach space, then the space $D_{T}(\Phi'_{q})$ is complete, separable and metrizable (see \cite{EthierKurtz, KallianpurXiong}). 

The next result characterizes the compact subsets of $D_{T}(\Phi'_{q})$. For its statement we will need the following \emph{modulus of continuity}:
\begin{enumerate}
\item If $x \in D_{T}(\Phi'_{q})$, $\delta>0$, let 
$$w'_{x}(\delta,q)=\inf_{\{t_{i}\}} \max_{1 \leq i \leq n} \sup \{ q'(x(t)-x(s)): s,t \in [t_{i-1},t_{i})\}, $$
\item If $x \in D_{T}(\Phi'_{q})$, $\phi \in \Phi_{q}$, $\delta>0$, let 
$$w'_{x}(\delta,\phi)=\inf_{\{t_{i}\}} \max_{1 \leq i \leq n} \sup \{ \abs{x(t)[\phi]-x(s)[\phi]}: s,t \in [t_{i-1},t_{i})\}, $$
\item If $x \in D_{T}(\R)$, $\delta >0$, let 
$$w'_{x}(\delta)= \inf_{\{t_{i}\}} \max_{1 \leq i \leq n} \sup \{ \abs{x(t)-x(s)}: s,t \in [t_{i-1},t_{i})\}, $$
\end{enumerate}  
where the infimum is taken over the finite partitions $0=t_{0} < t_{1}< \dots < t_{n}=T$, $t_{i}-t_{i-1}> \delta$, $i=1,2,  \dots, n$.  

\begin{prop} [\cite{KallianpurXiong}, Theorem 2.4.3] \label{propProperSkorHilbSpaces} Let $q$ be a continuous seminorm on $\Phi$. Then, $A \subseteq D_{T}(\Phi'_{q})$ is compact if and only if the following two conditions are satisfied:
\begin{enumerate}
\item there exists $\mathcal{K} \subseteq \Phi'_{q}$ compact such that $x(t) \in \mathcal{K}$  $\forall t \in [0,T], x \in A$.   
\item $\lim_{\delta \rightarrow 0+} \sup_{x \in A} w'_{x}(\delta,q )=0$.
\end{enumerate}
\end{prop}

Now let $\theta$ be a weaker countably Hilbertian topology on the nuclear space $\Phi$. We proceed to study some properties of the space $D_{T}( (\widetilde{\Phi_{\theta}})'_{\beta})$ equipped with its Skorokhod topology, that we will denote temporarily by $D_{s,T}((\widetilde{\Phi_{\theta}})'_{\beta})$ to distinguish it from the inductive limit topology that we introduce below. 

Let $(p_{n}: n \in \N)$ be an increasing sequence of continuous Hilbertian seminorms on $\Phi$ generating the topology $\theta$. Because (see \cite{FonsecaMora:2018}, Proposition 2.4)
$$(\widetilde{\Phi_{\theta}})'_{\beta} = \bigcup_{n \in \N} \Phi'_{p_{n}},$$
it is a consequence of the Banach-Steinhaus theorem that (see e.g. \cite{Jakubowski:1986}, Proposition 5.3) 
\begin{equation} \label{decompSkorSpaceWeakCountHilb}
D_{T}((\widetilde{\Phi_{\theta}})'_{\beta}) = \bigcup_{n \in \N} D_{T}(\Phi'_{p_{n}} ). 
\end{equation}
Moreover because the canonical inclusion from $\Phi'_{p_{n}}$ into $(\widetilde{\Phi_{\theta}})'_{\beta}$ is continuous, then for each $n \in \N$ the inclusion from $D_{T}(\Phi'_{p_{n}})$ into $D_{s,T}((\widetilde{\Phi_{\theta}})'_{\beta})$ is continuous (see Lemma 1.5 in \cite{Jakubowski:1986}). 

In view of \eqref{decompSkorSpaceWeakCountHilb} and following an idea from P\'{e}rez-Abreu and Tudor in \cite{PerezAbreuTudor:1992}, we can also consider on $D_{T}((\widetilde{\Phi_{\theta}})'_{\beta})$ the \emph{inductive limit topology} with respect to the spaces $(D_{T}(\Phi'_{p_{n}}): n \in \N)$, i.e. the finest topology for which the inclusions from $D_{T}(\Phi'_{p_{n}})$ into $D_{T}((\widetilde{\Phi_{\theta}})'_{\beta})$ are continuous. We denote the space $D_{T}((\widetilde{\Phi_{\theta}})'_{\beta})$ equipped with this topology by $D_{i,T}((\widetilde{\Phi_{\theta}})'_{\beta})$.  
We summarize properties of $D_{i,T}((\widetilde{\Phi_{\theta}})'_{\beta})$ and $D_{s,T}((\widetilde{\Phi_{\theta}})'_{\beta})$ in the following result.

\begin{prop} \label{propProperSkoroSpaceWeakerCHT}
Let $\theta$ be a weaker countably Hilbertian topology on the nuclear space $\Phi$. Then, 
\begin{enumerate}
\item The spaces $D_{i,T}((\widetilde{\Phi_{\theta}})'_{\beta})$ and $D_{s,T}((\widetilde{\Phi_{\theta}})'_{\beta})$ are Souslin.
\item The compact subsets of $D_{s,T}((\widetilde{\Phi_{\theta}})'_{\beta})$ are metrizable. 
\item The canonical inclusion from $D_{s,T}( (\widetilde{\Phi_{\theta}})'_{\beta})$ (and hence from $D_{i,T}((\widetilde{\Phi_{\theta}})'_{\beta})$) into $D_{T}(\Phi'_{\beta})$ is continuous.  
\end{enumerate}
\end{prop}
\begin{prf} 
\emph{(1)} The space $D_{i,T}((\widetilde{\Phi_{\theta}})'_{\beta})$, being the inductive limit of the Souslin spaces $D_{T}(\Phi'_{p_{n}})$ is again a Souslin space (see \cite{Treves}, Proposition A.4(c), p.551). Now, because the canonical inclusion from $D_{i,T}((\widetilde{\Phi_{\theta}})'_{\beta})$ into $D_{s,T}((\widetilde{\Phi_{\theta}})'_{\beta})$ is continuous, then it follows that $D_{s,T}((\widetilde{\Phi_{\theta}})'_{\beta})$ is also Souslin. 

\emph{(2)} Because $\widetilde{\Phi_{\theta}}$ is ultrabornological, hence barrelled, if $K \subseteq (\widetilde{\Phi_{\theta}})'_{\beta}$ is compact then it is equicontinuous 
(see \cite{Schaefer}, Theorem IV.5.2, p.141). Therefore, there exists a continuous Hilbertian seminorm $p$ on $\widetilde{\Phi_{\theta}}$ such that $K \subseteq B_{p'}(1)$. Consequently, the set $K$ is metrizable. Then, because each compact subset of $(\widetilde{\Phi_{\theta}})'_{\beta}$ is metrizable, the space $D_{s,T}( (\widetilde{\Phi_{\theta}})'_{\beta})$ inherits the same property (see Proposition 1.6.vii) in \cite{Jakubowski:1986}). 

\emph{(3)} The conclusion follows from the fact that the topology on $(\widetilde{\Phi_{\theta}})'_{\beta}$ is finner than the induced topology from $\Phi'_{\beta}$ and Lemma 1.5 in \cite{Jakubowski:1986}. 
\end{prf}

\begin{rema}
If $\Phi$ is a F\'{e}chet nuclear space and $\theta$ coincides with the nuclear topology on $\Phi$ (hence $\widetilde{\Phi_{\theta}}=\Phi$), it it shown in Proposition 3.1 in \cite{PerezAbreuTudor:1992} that $\mathcal{B}(D_{i,T}(\Phi'_{\beta}))= \mathcal{B}(D_{s,T}(\Phi'_{\beta}))$.  
\end{rema}

\begin{note}
From now on and unless otherwise specified we will always assume that $D_{T}( (\widetilde{\Phi_{\theta}})'_{\beta})$ and $D_{T}(\Phi'_{\beta})$ are equipped with their Skorokhod  topologies. %We will make a brief reference to the Skorokhod M1 topology in Section \ref{}.   
\end{note}

The next result gives characterizations for compact subsets of 
$D_{T}(\Phi'_{\beta})$ when $\Phi$ is a barrelled nuclear space. We will need the following definition: for each $\phi \in \Phi$, let $\Pi_{\phi}: D_{T}(\Phi'_{\beta}) \rightarrow D_{T}(\R)$ the \emph{space projection} given by $x \mapsto x[\phi]=\{ x(t)[\phi]\}_{t \in [0,T]}$. 

\begin{theo} \label{theoCharacCompacSets} Let $\Phi$ be a nuclear space and let $A \subseteq D_{T}(\Phi'_{\beta})$. Consider the following statements:
\begin{enumerate}
\item $A$ is  compact in $D_{T}(\Phi'_{\beta})$. 
\item For any $\phi \in \Phi$, the set $\Pi_{\phi}(A)=\{ x[\phi]: x \in A\}$ is compact in $D_{T}(\R)$. 
\item There exists a continuous Hilbertian seminorm $q$ on $\Phi$ such that $A$ is  compact in $D_{T}(\Phi'_{q})$.  
\end{enumerate}
Then, we have (1) $\Rightarrow$ (2), (3) $\Rightarrow$ (1),  and if $\Phi$ is a barrelled nuclear space, we have (2) $\Rightarrow$ (3). 
\end{theo}
\begin{prf} First observe from Proposition \ref{propProperSkoroSpaceWeakerCHT}(3) that the canonical inclusion $D_{T}(\Phi'_{q}) \rightarrow D_{T}(\Phi'_{\beta})$ is continuous, therefore if $A$ is compact in $D_{T}(\Phi'_{q})$, it is also in $D_{T}(\Phi'_{\beta})$. This shows \emph{(3) $\Rightarrow$ (1)}. 
Similarly, (1) $\Rightarrow$ (2) is a direct consequence of the continuity of the space projection $\Pi_{\phi}$ for each $\phi \in \Phi$. 

Now assume that $\Phi$ is a barrelled nuclear space. We are going to show that \emph{(2) $\Rightarrow$ (3)}. Let $K=\bigcup_{x \in A} \{ x(t): t \in [0,T]\}$.
Clearly, $A \subseteq D_{T}((K,\beta \cap K))$; where $(K,\beta \cap K)$ denotes the subspace $K \subseteq \Phi'_{\beta}$ equipped with the subspace topology induced on $K$ by the strong topology $\beta$ on $\Phi'$.
Moreover because for each $\phi \in \Phi$ the set $\Pi_{\phi}(A)$ is compact, then 
$$\sup_{x \in A} \sup_{t \in [0,T]} \abs{x(t)[\phi]} < \infty, \quad \forall \, \phi \in \Phi.$$ 
Therefore the set $K \subseteq \Phi'_{\beta}$ is weakly bounded and because $\Phi$ is barrelled, this implies that $K$ is strongly bounded and equicontinuous
(see \cite{Schaefer}, Theorem IV.5.2, p.141). Then the polar $K^{0}$ of $K$ is a neighborhood of zero of $\Phi$. But  because $\Phi$ is nuclear, there exists a continuous Hilbertian seminorm $p$ on $\Phi$ such that $B_{p}(1) \subseteq K^{0}$. If $p_{K^{0}}$ is the continuous seminorm on $\Phi$ with unit ball $K^{0}$ (i.e. $p_{K^{0}}$ is the Minkowski functional of $K^{0}$), we then have that the inclusion $i_{p_{K^{0}},p}: \Phi_{p} \rightarrow \Phi_{p_{K^{0}}}$ is continuous, and hence its dual operator $i'_{p_{K^{0}},p}: \Phi'_{p_{K^{0}}} \rightarrow \Phi_{p}$ is continuous. Let $q$ be a continuous Hilbertian seminorm on $\Phi$ such that $p \leq q$ and $i_{p,q}: \Phi_{q} \rightarrow \Phi_{p}$ is Hilbert-Schmidt. Then $i'_{p,q}: \Phi'_{p} \rightarrow \Phi'_{q}$ is Hilbert-Schmidt. But because $K$ is the unit ball in $\Phi'_{p_{K^{0}}}$, and the map $i'_{p_{K^{0}},q}=i'_{p,q} \circ i'_{p_{K^{0}},p}$ is Hilbert-Schmidt, hence compact, the image of $K$ under $i'_{p_{K^{0}},q}$ is relatively compact in $\Phi'_{q}$. Thus if $\mathcal{K}$ denote the closure of $K$ in $\Phi'_{q}$, we then have that $\mathcal{K}$ is compact in $\Phi'_{q}$ and that $A \subseteq D_{T}((K,\beta \cap K)) \subseteq D_{T}(\Phi'_{p_{K^{0}}}) \subseteq D_{T}(\Phi'_{q})$ (the inclusion being justified because the the topology on $\Phi'_{p_{K^{0}}}$  is finner that the topology induced by $\Phi'_{q}$; see \cite{Jakubowski:1986}, Lemma 1.5). Hence $A$ is a subset of $D_{T}(\Phi'_{q})$ that satisfies the first condition in Proposition \ref{propProperSkorHilbSpaces}. 

Our next objective is to show that $A$ also satisfies the second condition in Proposition \ref{propProperSkorHilbSpaces}. We will follow some ideas from the proof of Theorem 2.4.4 in \cite{KallianpurXiong}. 

Let $(\phi^{q}_{j})_{j \in \N} \subseteq \Phi$ be a complete orthonormal system in $\Phi_{q}$. Observe that for each $j \in \N$, the map  $\Pi_{\phi^{q}_{j}}:D_{T}(\Phi'_{\beta}) \rightarrow D_{T}(\R)$ given by $x \mapsto x[\phi^{q}_{j}]=\{ x(t)[\phi^{q}_{j}]: t \in [0,T]\}$ is continuous. Then for every $j \in \N$ the set $B_{j} \defeq \Pi_{\phi^{q}_{j}}(A)= \{ x[\phi^{q}_{j}]: x \in A\}$ is compact in $D_{T}(\R)$. Therefore we have that 
$$ \lim_{\delta \rightarrow 0+} \sup_{x \in A} w'_{x}(\delta,\phi^{q}_{j} ) \leq \lim_{\delta \rightarrow 0+} \sup_{y \in B_{j}} w'_{y}(\delta)=0.$$
Now, recall from our previous arguments that $x(t) \in \mathcal{K} \subseteq B_{p'}(1)$ for every $t \in [0,T], x \in A$. Then
$$ \sup_{x \in A} w'_{x}(\delta,\phi^{q}_{j} )^{2} \leq \sup_{x \in A} \sup_{t \in [0,T]} 4 \abs{x(t)[\phi^{q}_{j}]}^{2} \leq 4 p(\phi^{q}_{j})^{2}, $$
but because $i_{p,q}$ is Hilbert-Schmidt we have that $\sum_{j=1}^{\infty} p(\phi^{q}_{j})^{2} < \infty$. Therefore from the dominated convergence theorem, 
$$ \lim_{\delta \rightarrow 0+} \sup_{x \in \mathcal{K}} w'_{x}(\delta,q )^{2} \leq \sum_{j=1}^{\infty} \lim_{\delta \rightarrow 0+} \sup_{x \in \mathcal{K}} w'_{x}(\delta,\phi^{q}_{j} )^{2}=0. $$
Thus Proposition \ref{propProperSkorHilbSpaces} shows that $A$ is compact in  $D_{T}(\Phi'_{q})$.
\end{prf}

\begin{rema}
If $\Phi$ is a barrelled nuclear space and $(q_{i}: i \in I)$ is a family of continuous Hilbertian seminorms generating the nuclear topology on $\Phi$, then similarly as we did for \eqref{decompSkorSpaceWeakCountHilb} the Banach-Steinhaus theorem shows that 
$$ D_{T}(\Phi'_{\beta})=  \bigcup_{i \in I} D_{T}(\Phi'_{q_{i}} ),$$
and hence, we can also define the \emph{inductive limit topology} on $D_{T}(\Phi'_{\beta})$ with respect to the family of  spaces $(D_{T}(\Phi'_{q_{i}}): i \in I)$. Using Theorem \ref{theoCharacCompacSets} and from similar arguments to those used in \cite{PerezAbreuTudor:1992}, one can show that the compact subsets of $D_{T}(\Phi'_{\beta})$ coincide under both the inductive limit topology and the Skorokhod topology (see Lemma 3.2 in \cite{PerezAbreuTudor:1992} for the details).  
\end{rema}

\section{Measures and  Random Variables in  $D_{T}(\Phi'_{\beta})$} \label{sectionMeRVSkoSpac}

\begin{assu}
Unless otherwise indicated, in this section $\Phi$ will denote a (Hausdorff) locally convex space.
\end{assu}

\subsection{Cylindrical measures and cylindrical random variables}

In this section we introduce the concepts of cylindrical measures and cylindrical random variables in $D_{T}(\Phi'_{\beta})$. We want to stress the fact that the standard definitions for these objects cannot be directly formulated because the space $D_{T}(\Phi'_{\beta})$ is not a topological vector space since the addition is not continuous. The main motivation for the introduction of these two concepts is because they provide an alternative approach to handle the usual measurability problems that occurs on the space $D_{T}(\Phi'_{\beta})$ (see \cite{Jakubowski:1986}). 

We start by introducing the class of cylindrical sets. Let $\phi_{1}, \dots, \phi_{m} \in \Phi$, $t_{1}, \dots, t_{m} \in [0,T]$, $m \in \N$.
We define the \emph{space-time projection map} 
$\Pi^{\phi_{1}, \dots, \phi_{m}}_{t_{1}, \dots, t_{m}}: D_{T}(\Phi'_{\beta}) \rightarrow \R^{m}$ by 
\begin{equation*}
\Pi^{\phi_{1}, \dots, \phi_{m}}_{t_{1}, \dots, t_{m}}(x)  =
(x(t_{1})[\phi_{1}], \dots, x(t_{m})[\phi_{m}]), \quad \forall \, x \in D_{T}(\Phi'_{\beta}).
\end{equation*} 

If $M=\{\phi_{1}, \dots, \phi_{m} \} \subseteq  \Phi $, $I=\{ 
t_{1}, \dots, t_{m} \} \subseteq  [0,T]$ and $B \in \mathcal{B}(\R^{m})$, the set 
\begin{eqnarray*}
\mathcal{Z}(M,I,B) 
& \defeq & \left( \Pi^{\phi_{1}, \dots, \phi_{m}}_{t_{1}, \dots, t_{m}} \right)^{-1}(B) \\
&= & \{ x \in D_{T}(\Phi'_{\beta}): (x(t_{1})[\phi_{1}], \dots, x(t_{m})[\phi_{m}]) \in B\},
\end{eqnarray*}  
is called a \emph{cylinder set} in $D_{T}(\Phi'_{\beta})$ based on $(M,I)$. Moreover the collection 
$ \mathcal{C}(D_{T}(\Phi'_{\beta});M,I)=\{ \mathcal{Z}(M,I,B): B \in \mathcal{B}(\R^{m})\}$ is a $\sigma$-algebra, called the \emph{cylindrical $\sigma$-algebra in $D_{T}(\Phi'_{\beta})$ based on} $(M,I)$. Furthermore, we denote the collection of all the cylinder sets in $D_{T}(\Phi'_{\beta})$ by $\mathcal{Z}(D_{T}(\Phi'_{\beta}))$ and the $\sigma$-algebra they generate by 
$\mathcal{C}(D_{T}(\Phi'_{\beta}))$. We call $\mathcal{C}(D_{T}(\Phi'_{\beta}))$ the \emph{cylindrical $\sigma$-algebra} in $D_{T}(\Phi'_{\beta})$. 

One can easily check that we have the inclusion $\mathcal{C}(D_{T}(\Phi'_{\beta})) \subseteq  \mathcal{B}(D_{T}(\Phi'_{\beta}))$. But the converse is not true in general.  Nevertheless, the following result shows that if we consider $\Phi$ equipped with a weaker countably Hilbertian topology then the identity  holds. 

\begin{lemm} \label{lemmCylinAndBorelSigmaAlgCoincide}
If $\theta$ is a weaker countably Hilbertian topology on  $\Phi$, then $\mathcal{C}(D_{T}((\widetilde{\Phi_{\theta}})'_{\beta})) =  \mathcal{B}(D_{T}((\widetilde{\Phi_{\theta}})'_{\beta}))$. 
\end{lemm}
\begin{prf}
Let $(p_{n}:n \in \N)$ be an increasing sequence of continuous Hilbertian seminorms on $\Phi$ that generates the topology $\theta$ on $\Phi$. From \eqref{decompSkorSpaceWeakCountHilb} and because for each $n \in \N$, $\mathcal{C}(D_{T}(\Phi'_{p_{n}})) =  \mathcal{B}(D_{T}(\Phi'_{p_{n}}))$ (this is a consequence of the fact that each $\Phi'_{p_{n}}$ is separable and metric; see Corollary 2.4 in \cite{Jakubowski:1986}), it follows that $\mathcal{C}(D_{T}((\widetilde{\Phi_{\theta}})'_{\beta})) =  \mathcal{B}(D_{T}((\widetilde{\Phi_{\theta}})'_{\beta}))$. 
\end{prf}

\begin{defi} A \emph{cylindrical (probability) measure} on $D_{T}(\Phi'_{\beta})$ is a map $\mu: \mathcal{Z}(D_{T}(\Phi'_{\beta})) \rightarrow [0,+\infty]$ such that for each finite $M \subseteq \Phi$ and $I \subseteq [0,T]$, the restriction of $\mu$ to  $\mathcal{C}(D_{T}(\Phi'_{\beta});M,I)$ is a (probability) measure. 
\end{defi}

Let $\mu$ be a cylindrical probability measure on $D_{T}(\Phi'_{\beta})$. For $t \in [0,T]$, consider the \emph{time projection} $\Pi_{t}: D_{T}(\Phi'_{\beta}) \rightarrow \Phi'_{\beta}$ by $x \mapsto x(t)$. We define the \emph{Fourier transform} $\widehat{\mu}_{t}$ of the measure $\mu$ at time $t$ as the Fourier transform of the measure $\mu_{t} \defeq \mu \circ \Pi_{t}^{-1}$ on $\Phi'_{\beta}$, i.e. the function $\widehat{\mu}_{t}: \Phi \rightarrow \C$ is given by 
$$ \widehat{\mu}_{t}(\phi)= \int_{D_{T}(\Phi'_{\beta})} \, e^{ix(t)[\phi]} d\mu, \quad \forall \, \phi \in \Phi. $$

Now, note that as $\mathcal{C}(D_{T}(\Phi'_{\beta})) \subseteq  \mathcal{B}(D_{T}(\Phi'_{\beta}))$, a Borel probability measure on $D_{T}(\Phi'_{\beta})$ clearly defines a cylindrical probability measure on $D_{T}(\Phi'_{\beta})$. When $\Phi$ is a nuclear space, sufficient conditions for a cylindrical probability measure on $D_{T}(\Phi'_{\beta})$ to extend to a Borel probability measure on $D_{T}(\Phi'_{\beta})$ in terms of its Fourier transforms will be given in Theorem \ref{theoMinlosSkorokhodSpace}. 

\begin{defi} A \emph{cylindrical random variable} in $D_{T}(\Phi'_{\beta})$ (defined on a probability space $\ProbSpace$) is a linear map $X: \Phi \rightarrow L^{0}(\Omega, \mathcal{F}, \Prob;D_{T}(\R))$, where $L^{0}(\Omega, \mathcal{F}, \Prob;D_{T}(\R))$ is the set of $D_{T}(\R)$-valued random variables defined on $\ProbSpace$. 
\end{defi} 

If $\mathcal{Z}(M,I,B)$ is a cylinder set in $D_{T}(\Phi'_{\beta})$ with $M=\{ \phi_{1}, \dots, \phi_{m}\} \subseteq \Phi$, $I=\{ t_{1}, \dots, t_{m}\} \subseteq [0,T]$ and $B \in \mathcal{B}(R^{m})$,  let 
\begin{eqnarray*}
\mu_{X}(\mathcal{Z}(M,I,B)) 
& \defeq & \Prob \left( ( X(\phi_{1})(t_{1}), \dots, X(\phi_{m})(t_{m}) \in B  \right) \\
& = & \Prob \circ X^{-1} \circ \left( \Pi^{\phi_{1},\dots, \phi_{m}}_{t_{1}, \dots, t_{m}}\right)^{-1}(B). 
\end{eqnarray*} 
The map $\mu_{X}$ is called the \emph{cylindrical distribution} of $X$ and it is a  cylindrical probability measure on $D_{T}(\Phi'_{\beta})$. The Fourier transform $\widehat{\mu}_{X,t}$ at time $t$ of the cylindrical random variable $X$ in $D_{T}(\Phi'_{\beta})$ is that of its cylindrical measure $\mu_{X}$ at time $t$. 

As the next results shows,  to every cylindrical probability measure on $D_{T}(\Phi'_{\beta})$ there corresponds a canonical  cylindrical random variable in $D_{T}(\Phi'_{\beta})$.  

\begin{theo}\label{theoExistCanoCylRandVariable}
Let $\mu$ be a cylindrical probability measure on $D_{T}(\Phi'_{\beta})$. Then, there exists a cylindrical random variable $X$ in $D_{T}(\Phi'_{\beta})$ defined on some probability space $\ProbSpace$ whose cylindrical distribution is $\mu$. 
\end{theo}
\begin{prf}
For the proof we will construct a compatible family of measures that satisfies the Kolmogorov Extension Theorem. 

For each $\phi \in \Phi$, let $(\Omega_{\phi},\mathscr{F}_{\phi})$ be the Borel space $(D_{T}(\R), \mathcal{B}(D_{T}(\R)))$. For any $F \subseteq \Phi$, we write $\Omega^{F}=\times_{\phi \in F} \Omega_{\phi}$ and $\mathscr{F}^{F}=\otimes_{\phi \in F} \mathscr{F}_{\phi}$. Furthermore, we define $\pi_{F}: D_{T}(\Phi'_{\beta}) \rightarrow \Omega^{F}$ by $\pi_{F}(x)=(x[\phi])_{\phi  \in F}$ $\forall x \in D_{T}(\Phi'_{\beta})$. For finite $F=\{ \phi_{1}, \dots, \phi_{m} \} \subseteq \Phi$, we also use the notation $\pi_{\phi_{1}, \dots, \phi_{m}}$ for $\pi_{F}$. 
If $G \subseteq F \subseteq \Phi$, we denote by $\pi_{F,G}$ the map that takes $y \in \Omega^{F}$ into its restriction to $G$.
Clearly, the maps $\pi_{F}$ and $\pi_{F,G}$ are measurable.

Fix $F=\{ \phi_{1},\dots, \phi_{m} \} \subseteq \Phi$. For every $t_{1}, \dots, t_{m}  \in [0,T]$, define $\pi_{t_{1}, \dots, t_{m}}: \Omega^{F} \rightarrow \R^{m}$ by $\pi_{t_{1}, \dots, t_{m}}(y)=(y_{1}(t_{1}), \dots, y_{m}(t_{m}))$ for all $y=(y_{1}, \dots, y_{m}) \in \Omega^{F}$. It is a well-known fact that $\mathcal{B}(D_{T}(\R))= \sigma( \{ y \in D_{T}(\R): y(t) \in B\}: t \in [0,T], B \in \mathcal{B}(\R))$, hence it is clear that the family of cylinder sets $\{ (\pi_{t_{1}, \dots, t_{m}})^{-1} (B_{1} \times \dots \times B_{m}): t_{1}, \dots, t_{m} \in [0,T]$, $B_{1}, \dots, B_{m} \in \mathcal{B}(\R) \}$ in $\Omega^{F}$ generates $\mathscr{F}^{F}$. 

Now, define $\mu_{F}$ on $\Omega^{F}$ as follows: first, for $I=\{ t_{1}, \dots, t_{m} \} \subseteq [0,T]$, $B_{1}, \dots, B_{m} \in \mathcal{B}(\R)$, let 
\begin{equation} \label{defiMeasureMuFExtenTheorem}
\mu_{F} ((\pi_{t_{1}, \dots, t_{m}})^{-1} (B_{1} \times \dots \times B_{m}))=\mu(\mathcal{Z}(F,I,B_{1} \times \dots \times B_{m})).
\end{equation}
As the cylinder sets in $\Omega^{F}$ generates $\mathscr{F}^{F}$, the above definition of $\mu_{F}$ extends to a probability measure (that we denote again by $\mu_{F}$) on $(\Omega^{F}, \mathscr{F}^{F})$. 

Now, we will show that the measures $(\mu_{F}: F \subseteq \Phi, \, F \mbox{ finite})$ satisfy the consistency condition. We start by showing it for the cylinder sets. 
Let $G =\{ \varphi_{1}, \dots, \varphi_{n}\} \subseteq F=\{\phi_{1}, \dots, \phi_{m}\} \subseteq \Phi $, and consider $I=\{ t_{1}, \dots, t_{n}\} \subseteq J=\{s_{1}, \dots, s_{m} \} \subseteq [0,T]$ and $B_{1}, \dots, B_{n} \in \mathcal{B}(\R)$. For $i=1, \dots, n$, let $s_{j_{1}}, \dots, s_{j_{n}}$ given by $s_{j_{i}}=t_{i}$, and for $j=1, \dots, m$, let $A_{j} = \R$ if $s_{j} \notin \{ s_{j_{1}}, \dots, s_{j_{n}} \}$ and $A_{j} =B_{j}$ if $s_{j} \in \{ s_{j_{1}}, \dots, s_{j_{n}} \}$. Then, we have
\begin{flalign*}
& \mu_{G} ((\pi_{s_{j_{1}}, \dots, s_{j_{n}}})^{-1} (B_{1} \times \dots \times B_{n})) \\
& =  \mu (\mathcal{Z}(G,I,B_{1} \times \dots \times B_{n}))\\
& =  \mu (\mathcal{Z}(F,J,A_{1} \times \dots \times A_{m}))\\
& =  \mu_{F} ((\pi_{s_{1}, \dots, s_{m}})^{-1} (A_{1} \times \dots \times A_{m})) \\
& =  \mu_{F} (\pi_{F,G}^{-1} ((\pi_{s_{j_{1}}, \dots, s_{j_{n}}})^{-1} (B_{1} \times \dots \times B_{n}))).
\end{flalign*}
Now, because the above equality holds for any cylinder set, then it also holds for any set in $\mathscr{F}^{G}$, showing that the consistency condition $\mu_{G} = \mu_{F} \circ \pi_{F,G}^{-1} $ is satisfied for $G \subseteq F \subseteq \Phi$ with $G$ and $F$ finite. 

Therefore, by the Kolmogorov's extension theorem (see \cite{Parthasarathy}, Theorem 5.1, p.144), there exists a unique probability measure $\Prob$ on $(\Omega^{\Phi},\mathcal{F}^{\Phi})\defeq (\times_{\phi \in \Phi} D_{T}(\R), \otimes_{\phi \in  \Phi} \mathcal{B}(D_{T}(\R)))$ such that for each finite $F \subseteq \Phi$, 
\begin{equation} \label{defiProbMeasCanoCyliRV}
\mu_{F}(\Gamma)=\Prob (\pi_{F}^{-1} (\Gamma)), \quad \forall \, \Gamma \in \mathcal{F}^{F}.
\end{equation} 
In particular, for every $F=\{\phi_{1}, \dots, \phi_{n}\} \subseteq \Phi $, $I=\{ t_{1}, \dots, t_{n}\} \subseteq [0,T]$ and $B_{1}, \dots, B_{n} \in \mathcal{B}(\R)$, from the above equality and \eqref{defiMeasureMuFExtenTheorem} we have that
\begin{eqnarray}
\mu(Z(F,I,B_{1} \times \dots \times B_{n}))
& = & \mu_{F}(\pi_{t_{1},\dots, t_{n}}^{-1} (B_{1} \times \dots \times B_{n})) \nonumber \\
& = &\Prob ( \pi_{\phi_{1}, \dots, \phi_{n}}^{-1} (\pi_{t_{1},\dots, t_{n}}^{-1} (B_{1} \times \dots \times B_{n}))) \nonumber \\
& = &\Prob ( (\Pi^{\phi_{1}, \dots, \phi_{n}}_{t_{1},\dots, t_{n}})^{-1} (B_{1} \times \dots \times B_{n})) \label{extensionMeasureOnCylinders}  
%& = & \Prob (y \in D_{T}(E'_{\beta}): y[\phi_{1}](t_{1}) \in B_{1}, \dots, y[\phi_{n}](t_{n}) \in B_{n}). 
\end{eqnarray} 
Our next step is to define a cylindrical random variable in $D_{T}(\Phi'_{\beta})$ whose cylindrical distribution is $\mu$. Let $X: \Phi \rightarrow L^{0}(\Omega^{\Phi}, \mathcal{F}^{\Phi}, \Prob;D_{T}(\R))$ be defined in the following way: for every $\phi \in \Phi$, let $X(\phi) \defeq \pi_{\phi}$, i.e. $X(\phi)(y)= y(\phi) \in D_{T}(\R)$ for each $y = (y(\phi))_{\phi \in \Phi} \in \Omega^{\Phi}$. Clearly, each $X(\phi)$ is $\mathscr{F}^{\Phi}/\mathcal{B}(D_{T}(\R))$-measurable and therefore $X$ is well-defined. Moreover, the fact that $\mu$ is the cylindrical distribution of $X$ is a direct consequence of \eqref{extensionMeasureOnCylinders}.  

Now we show  that $X$ is linear. Let $\psi_{1}, \psi_{2} \in \Phi$ and $\lambda_{1}, \lambda_{2} \in \R$. Consider the subset $F=\{ \lambda_{1} \psi_{1},  \lambda_{2} \psi_{2}, \lambda_{1} \psi_{1}+ \lambda_{2} \psi_{2}\}$ of $\Phi$. Then $\Omega^{F}=D_{T}(\R)^{3}$ and $\pi_{F}: D_{T}(\Phi'_{\beta}) \rightarrow D_{T}(\R)^{3}$ is given by $x \mapsto  (\lambda_{1} x[\psi_{1}], \lambda_{2} x[\psi_{2}],x[\lambda_{1} \psi_{1}+ \lambda_{2} \psi_{2}])$. If $\sigma: D_{T}(\R)^{3} \rightarrow D_{T}(\R)$ is defined as $(u,v,w) \mapsto u+v-w$, then $\sigma$ is continuous and also $\sigma \circ \pi_{F}=0 \in D_{T}(\R)$. Thus for $A \in \mathcal{B}(D_{T}(\R))$, $\mu_{F} ( \sigma^{-1}(A))= \mu \circ  \pi_{F}^{-1}( \sigma^{-1}(A))$ takes value $0$ if $0 \notin A$ and takes value $1$ if $0 \in A$. Hence $\mu_{F} $ is supported by the plane $\sigma^{-1}(\{ 0 \})=\{ (u,v,w): u+v-w=0\}$ of $\Omega^{F}=D_{T}(\R)^{3}$. But then  we have from \eqref{defiProbMeasCanoCyliRV} that
\begin{flalign*}
& \Prob \left( \lambda_{1} X(\psi_{1})+ \lambda_{2} X(\psi_{2})-  X(\lambda_{1}\psi_{1}+\lambda_{2}\psi_{2}) = 0  \right) \\
& = \Prob \left( (\lambda_{1} X(\psi_{1}), \lambda_{2} X(\psi_{2}),  X(\lambda_{1}\psi_{1}+\lambda_{2}\psi_{2}) ) \in \sigma^{-1}(\{ 0 \}) \right)\\
& = \mu_{F} ( \sigma^{-1}(\{ 0 \}))=1.
\end{flalign*}
This proves that $X$ is a cylindrical random variable in $D_{T}(\Phi'_{\beta})$ defined on $(\Omega^{\Phi}, \mathcal{F}^{\Phi}, \Prob)$ with cylindrical measure $\mu$.
\end{prf}

%\begin{rema}
%We are not aware of any previous work that considers the  concepts of cylindrical measures and cylindrical random variables in the Skorokhod space introduced in this section. Usually, the cylinder sets in the Skorokhod space are defined with respect to time projections only (see e.g. \cite{EthierKurtz, Jakubowski:1986}). However, the use of space-time projections and the cylindrical $sigma$-algebra they generate have proven useful tools for the case of duals of nuclear spaces. This is the case of the original work of Mitoma \cite{Mitoma:1983} and the later work of Fouque \cite{Fouque:1984} for the J1 topology, and the more recent work of Ledger \cite{Ledger:2016} for the M1 topology.
%\end{rema}

\subsection{Measurability of random elements in $D_{T}(\Phi'_{\beta})$}

\begin{assu}
In this section and unless otherwise specified, all the random elements will be defined on a given probability space $\ProbSpace$. 
\end{assu}

We start this section by enumerating the relationship between the different types of random elements in $D_{T}(\Phi'_{\beta})$. 
\begin{enumerate}
\item Let $X$ be a $D_{T}(\Phi'_{\beta})$-valued random variable. Clearly $X$  determines a $\Phi'_{\beta}$-valued c\`{a}dl\`{a}g process $\{ X_{t} \}_{t \in [0,T]} $ given by  $X_{t}(\omega) \defeq X(\omega)(t)$ $\forall \,  t \in [0,T], \omega \in \Omega$. 
\item If $X$ is a cylindrical random variable in $D_{T}(\Phi'_{\beta})$, the linear map  $\phi \mapsto \{ X(\phi)(t) \}_{t \in [0,T]}$ is a cylindrical processes in $\Phi'$. Conversely, if $X=\{ X_{t} \}_{t \in [0,T]}$ is a cylindrical process such that for each $\phi \in \Phi$ the real-valued process $X(\phi)=\{ X_{t}(\phi) \}_{t \in [0,T]}$ is c\`{a}dl\`{a}g, then $X$ is a cylindrical random variable in $D_{T}(\Phi'_{\beta})$. 
\item Let $X=\{X_{t}\}_{t \in [0,T]}$ be a $\Phi'_{\beta}$-valued c\`{a}dl\`{a}g process. In this case $X$ defines two objects. First, for every $\phi \in \Phi$, the \emph{space projection} $\Pi_{\phi}: D_{T}(\Phi'_{\beta}) \rightarrow D_{T}(\R)$ maps $X$ into $X[\phi]= \Pi_{\phi} (X) \defeq \{ X_{t}[\phi]\}_{t \in [0,T]}$. This way, $X$ defines a cylindrical random variable in $D_{T}(\Phi'_{\beta})$.  Second, $X$ defines a map $X:\Omega \rightarrow D_{T}(\Phi'_{\beta})$ by means of its paths $\omega \mapsto (t \mapsto X(\omega)(t) \defeq X_{t}(\omega))$. This map is $\mathcal{F}/\mathcal{C}(D_{T}(\Phi'_{\beta}))$-measurable. 
However, because the inclusion
$\mathcal{C}(D_{T}(\Phi'_{\beta})) \subseteq  \mathcal{B}(D_{T}(\Phi'_{\beta}))$ might be strict, then $X$ is not necessarily a $D_{T}(\Phi'_{\beta})$-valued random variable. 
\end{enumerate}

Given a cylindrical process $X$ in $\Phi'$, the next result  gives sufficient conditions for the existence of a version  that is a $D_{T}(\Phi'_{\beta})$-valued random variable. 

\begin{theo}[Regularization theorem on Skorokhod space]  \label{theoRegulaTheoSkorokSpace}
Let $\Phi$ be a nuclear space. Let $X=\{X_{t} \}_{t \in [0,T]}$ be a cylindrical process in $\Phi'$ (e.g. a $\Phi'_{\beta}$-valued process) such that: 
\begin{enumerate}
\item For each $\phi \in \Phi$, the real-valued process $X(\phi)=\{ X_{t}(\phi) \}_{t \in [0,T]}$ has a c\`{a}dl\`{a}g version.
\item The family $\{ X_{t}: t \in [0,T] \}$ of linear maps from $\Phi$ into $L^{0} \ProbSpace$ is equicontinuous.  
\end{enumerate}
Then, there exist a weaker countable Hilbertian topology $\theta$ on $\Phi$ and a  $D_{T}((\widetilde{\Phi_{\theta}})'_{\beta})$-valued random variable $Y$ such that for each $\phi \in \Phi$ the real-valued c\`{a}dl\`{a}g processes $X(\phi)$ and $Y[\phi]$ are indistinguishable. 
In particular, $Y$ is a $D_{T}(\Phi'_{\beta})$-valued random variable whose probability distribution is a Radon measure on $D_{T}(\Phi'_{\beta})$. 
\end{theo}     
\begin{prf} From the properties \emph{(1)} and \emph{(2)} of $X$ and Theorem \ref{theoRegularizationTheoremCadlagContinuousVersion}, there exist a countably Hilbertian topology $\theta$ on $\Phi$ and a $(\widetilde{\Phi_{\theta}})'_{\beta}$-valued c\`{a}dl\`{a}g process $Y= \{ Y_{t} \}_{t \geq 0}$ that is a version of $X$ (unique up to indistinguishable versions). 

The mapping $\omega \mapsto Y_{t}(\omega)$ from $\Omega$ into $D_{T}((\widetilde{\Phi_{\theta}})'_{\beta})$ is $\mathcal{F}/\mathcal{C}(D_{T}((\widetilde{\Phi_{\theta}})'_{\beta})$-measurable. But from Lemma \ref{lemmCylinAndBorelSigmaAlgCoincide} it is also 
$\mathcal{F}/\mathcal{B}(D_{T}((\widetilde{\Phi_{\theta}})'_{\beta}))$-measurable. Thus $Y$ defines a $D_{T}((\widetilde{\Phi_{\theta}})'_{\beta})$-valued random variable and its probability distribution on $D_{T}((\widetilde{\Phi_{\theta}})'_{\beta})$ is Radon because this space is Suslin (Proposition \ref{propProperSkoroSpaceWeakerCHT}(1)). Finally, because the inclusion map from $D_{T}((\widetilde{\Phi_{\theta}})'_{\beta})$ into $D_{T}(\Phi'_{\beta})$ is continuous (Proposition \ref{propProperSkoroSpaceWeakerCHT}(3)) then $Y$ is also a $D_{T}(\Phi'_{\beta})$-valued random variable whose probability distribution on $D_{T}(\Phi'_{\beta})$ is Radon. 
\end{prf}

As the following result shows, we can relax some of the conditions in Theorem \ref{theoRegulaTheoSkorokSpace} when the space $\Phi$ is ultrabornological. 

\begin{coro}\label{coroConditionRandVariSkoSpaceUltrab}
Assume that $\Phi$ is an ultrabornological nuclear space. Let $X=\{ X_{t}\}_{t \in [0,T]}$ be a $\Phi'_{\beta}$-valued process such that: 
\begin{enumerate}
\item For each $\phi \in \Phi$, the real-valued process $X[\phi]=\{ X_{t}[\phi] \}_{t \geq 0}$ has a c\`{a}dl\`{a}g version.
\item For each $t \in [0,T]$, the probability distribution $\mu_{t}$ of $X_{t}$ is a Radon measure on $\Phi'_{\beta}$.  
\end{enumerate}
Then, there exist a weaker countable Hilbertian topology $\theta$ on $\Phi$ and a  $(\widetilde{\Phi_{\theta}})'_{\beta}$-valued c\`{a}dl\`{a}g process $Y= \{ Y_{t} \}_{t \geq 0}$ that is a version of $X$ and is such that $Y$ is also a $D_{T}((\widetilde{\Phi_{\theta}})'_{\beta})$-valued random variable. In particular, $Y$ is a $D_{T}(\Phi'_{\beta})$-valued random variable whose probability distribution is a Radon measure on $D_{T}(\Phi'_{\beta})$. 
\end{coro}
\begin{prf}
First, from the fact that each  $\mu_{t}$ is a Radon measure and because the space $\Phi$ is barreled (this is because it is  ultrabornological, see \cite{NariciBeckenstein} p.449) it follows from Theorem 2.9 in \cite{FonsecaMora:2018} that each of the maps $X_{t}$ from $\Phi$ into $L^{0} \ProbSpace$ is continuous. Moreover because $\Phi$ is ultrabornological, the above property together with \emph{(1)} implies that the linear mapping from $\Psi$ into $D_{T}(\R)$ (equipped with supremum norm) given by $\psi \mapsto \{ X_{t}[\psi] \}_{t \in [0,T]}$ is continuous (see \cite{FonsecaMora:2018}, Proposition 3.10). This in particular shows that the family $\{ X_{t}: t \in [0,T] \}$ of linear maps from $\Phi$ into $L^{0} \ProbSpace$ is equicontinuous. The result then follows from Theorem \ref{theoRegulaTheoSkorokSpace}     
\end{prf}

As an important consequence of Theorem \ref{theoRegulaTheoSkorokSpace} we get the following interesting result concerning sufficient conditions for Radon extensions of cylindrical measures on $D_{T}(\Phi'_{\beta})$. 

 \begin{theo}[Minlos theorem on Skorokhod Space] \label{theoMinlosSkorokhodSpace} Let $\Phi$ be a nuclear space and let $\mu$ be a cylindrical probability measure on $D_{T}(\Phi'_{\beta})$. Suppose that the family of its Fourier transforms $(\widehat{\mu}_{t}: t \in [0,T])$ is equicontinuous at zero. Then, there exists a Radon probability measure $\nu$ on $D_{T}(\Phi'_{\beta})$ that is an extension of $\mu$. Moreover, there  exists a weaker countably Hilbertian topology $\theta$ on $\Phi$ such that $\nu$ is a Radon measure on $D_{T}((\widetilde{\Phi_{\theta}})'_{\beta})$. 
\end{theo}
\begin{prf}
First, from Theorem \ref{theoExistCanoCylRandVariable} there exists a cylindrical random variable $X$ in $D_{T}(\Phi'_{\beta})$ defined on some probability space $\ProbSpace$ whose cylindrical distribution is $\mu$. In particular, $\phi \mapsto \{ X(\phi)(t) \}_{t \in [0,T]}$ is a cylindrical processes in $\Phi'$ satisfying condition (1) in Theorem \ref{theoRegulaTheoSkorokSpace}.  

To check the second condition in Theorem \ref{theoRegulaTheoSkorokSpace}, observe that from the inequality (see \cite{Kallenberg}, Lemma 5.1, p.85)
$$ \Prob (\abs{X(\phi)(t)} \geq \epsilon) 
\leq \frac{\epsilon}{2} \int^{2/\epsilon}_{-2/\epsilon} (1-\Exp e^{i s X(\phi)(t)}) ds 
= \frac{\epsilon}{2} \int^{2/\epsilon}_{-2/\epsilon} (1-\widehat{\mu}_{t}(s\phi)) ds $$
valid for every  $\epsilon >0$, $t \in [0,T]$, and $\phi \in \Phi$, it follows that the equicontinuity of $(\widehat{\mu}_{t}: t \in [0,T])$ at zero implies that of the family of linear maps $\phi \mapsto X(\phi)(t)$, $ t \in [0,T]$, from $\Phi$ into $L^{0} \ProbSpace$. Hence by Theorem \ref{theoRegulaTheoSkorokSpace}
there exist a weaker countable Hilbertian topology $\theta$ on $\Phi$ and a $D_{T}((\widetilde{\Phi_{\theta}})'_{\beta})$-valued random variable $Y$ such that $Y[\phi]=X(\phi)$ $\Prob$-a.e. 

Let $\nu$ be the probability distribution of $Y$ on $D_{T}(\Phi'_{\beta})$. We know from Theorem \ref{theoRegulaTheoSkorokSpace} that $\nu$ is a Radon measure on $D_{T}((\widetilde{\Phi_{\theta}})'_{\beta})$, and hence a Radon measure on $D_{T}(\Phi'_{\beta})$. Moreover, for every $m \in \N$, $t_{1}, \dots, t_{m} \in [0,T]$, $\phi_{1}, \dots, \phi_{m} \in \Phi$ and $A \in \mathcal{B}(\R^{m})$, we have
\begin{eqnarray*}
\nu \left( \left( \Pi^{\phi_{1}, \dots, \phi_{m}}_{t_{1}, \dots, t_{m}} \right)^{-1} (A) \right) 
& = & \Prob \left( Y \in \left( \Pi^{\phi_{1}, \dots, \phi_{m}}_{t_{1}, \dots, t_{m}} \right)^{-1} (A) \right)\\
& = & \Prob \left( (Y(t_{1})[\phi_{1}], \dots,Y(t_{m})[\phi_{m}]) \in   A \right)\\
& = & \Prob \left( (X(\phi_{1})(t_{1}), \dots, X(\phi_{m})(t_{m}) \in   A \right)\\
& = & \mu \left( \left( \Pi^{\phi_{1}, \dots, \phi_{m}}_{t_{1}, \dots, t_{m}} \right)^{-1} (A) \right).
\end{eqnarray*}
Thus $\nu$ is an extension of $\mu$ as both measures agree on the cylindrical $\sigma$-algebra $\mathcal{C}(D_{T}(\Phi'_{\beta}))$ of $ D_{T}(\Phi'_{\beta})$.  
\end{prf}

As the next results shows, to every probability measure on $D_{T}(\Phi'_{\beta})$ with equicontinuous Fourier transforms there corresponds a canonical random variable in $D_{T}(\Phi'_{\beta})$.
Its validity is a direct consequence of the proof of Theorem \ref{theoMinlosSkorokhodSpace}. 

\begin{coro}\label{coroCanonicalRV}
Let $\Phi$ be a nuclear space. Suppose that $\mu$ is a probability measure on $D_{T}(\Phi'_{\beta})$ for which the family of its Fourier transforms $(\widehat{\mu}_{t}: t \in [0,T])$ is equicontinuous at zero. Then there  exists a weaker countably Hilbertian topology $\theta$ on $\Phi$ and a $D_{T}((\widetilde{\Phi_{\theta}})'_{\beta})$-valued random variable $Y$ whose probability distribution is $\mu$. 
\end{coro}

\section{Tightness of Probability Measures on the Skorokhod Space}\label{sectionTightSkorSpa}

\begin{assu}
Unless otherwise indicated, in this section we will always assume that $\Phi$ is a nuclear space. 
\end{assu}

\subsection{Uniform tightness on $D_{T}(\Phi'_{\beta})$} 

The main result of this section is  the following theorem that provides necessary and sufficient conditions for a family of probability measures $(\mu_{\alpha}: \alpha \in A)$ on $D_{T}(\Phi'_{\beta})$ to be uniformly tight. 

\begin{theo} \label{theoThighnessMeasures}
Let $(\mu_{\alpha}: \alpha \in A)$ be a family of probability measures on $D_{T}(\Phi'_{\beta})$ such that it satisfies the following:
\begin{enumerate}
\item \label{condEquiFouTRans} The family of Fourier transforms $(\widehat{\mu}_{\alpha,t}: t \in [0,T], \alpha \in A)$ is equicontinuous at zero.
\item \label{condWeakTight} For each $\phi \in \Phi$, the family $(\mu_{\alpha} \circ \Pi_{\phi}^{-1}: \alpha \in A)$ of probability measures on $D_{T}(\R)$ is uniformly tight.    
\end{enumerate}  
Then there exists a weaker countably Hilbertian topology $\theta$ on $\Phi$ such that $(\mu_{\alpha}: \alpha \in A)$ is uniformly tight on $D_{T}((\widetilde{\Phi_{\theta}})'_{\beta})$. In particular, the family $(\mu_{\alpha}: \alpha \in A)$ is uniformly tight on $D_{T}(\Phi'_{\beta})$. 

Conversely, if $\Phi$ is a barrelled nuclear space and the family $(\mu_{\alpha}: \alpha \in A)$ is uniformly tight on $D_{T}(\Phi'_{\beta})$, then conditions \ref{condEquiFouTRans} and \ref{condWeakTight} are satisfied. 
\end{theo} 

We proceed to prove Theorem \ref{theoThighnessMeasures}. The first step of the proof is the following result. 

\begin{prop} \label{propFamilyMeasRadonInCHS}
Let $(\mu_{\alpha}: \alpha \in A)$ be a family of  probability measures on $D_{T}(\Phi'_{\beta})$ satisfying condition \ref{condEquiFouTRans} in Theorem \ref{theoThighnessMeasures}. Then there exists a weaker countably Hilbertian topology $\theta$ on $\Phi$ such that for each $\alpha \in A$, $\mu_{\alpha}$ is a Radon probability measure on $D_{T}((\widetilde{\Phi_{\theta}})'_{\beta})$.
\end{prop}
\begin{prf}
For each $\alpha \in A$, denote by $X^{\alpha}$ the canonical cylindrical random variable in $D_{T}(\Phi'_{\beta})$ associated to $\mu_{\alpha}$ by  Theorem \ref{theoExistCanoCylRandVariable}. Without loss of generality we can assume that the $X^{\alpha}$'s are defined on the same probability space $\ProbSpace$. 

The same arguments in the proof of Theorem \ref{theoMinlosSkorokhodSpace} and the equicontinuity of the family $(\widehat{\mu}_{\alpha,t}: t \in [0,T], \alpha \in A)$ at zero shows that we can choose a weaker countably Hilbertian topology $\theta$ on $\Phi$ (independently of $\alpha$) and for each $\alpha \in A$  a $D_{T}((\widetilde{\Phi_{\theta}})'_{\beta})$-valued random variable $Y^{\alpha}$ such that $Y^{\alpha}(t)[\phi]=X^{\alpha}(\phi)(t)$ $\Prob$-a.e. 
Hence, following the same arguments to those used in the last part of the proof of Theorem \ref{theoMinlosSkorokhodSpace} it can be shown that $\mu_{\alpha}$ is a  Radon probability measure on $D_{T}((\widetilde{\Phi_{\theta}})'_{\beta})$ for each $\alpha \in A$
\end{prf}

The importance of Proposition \ref{propFamilyMeasRadonInCHS} is that it settle our problem  in the context of measures on $D_{T}((\widetilde{\Phi_{\theta}})'_{\beta})$. This fact is to be used in combination with the following result whose proof will be given at the end of this section. 

\begin{prop}\label{propSuffiCondTightOnUltrabor} Let $(\mu_{\alpha}: \alpha \in A)$ be a family of probability measures on $D_{T}(\Psi'_{\beta})$ where $\Psi$ is an ultrabornological space, and suppose that $\forall \phi \in \Psi$ the family $(\mu_{\alpha} \circ \Pi_{\phi}^{-1}: \alpha \in A)$ is uniformly tight on $D_{T}(\R)$. Then $\forall \epsilon >0$ there exists a continuous seminorm $p$ on $\Psi$ such that 
$$ \sup_{\alpha \in A} \int_{D_{T}(\Psi'_{\beta})} \, \sup_{t \in [0,T]} \abs{1-e^{ix(t)[\phi]}} d\mu_{\alpha} \leq \epsilon, \quad \forall \, \phi \in B_{p}(1). $$
In particular, the family of Fourier transforms $(\widehat{\mu}_{\alpha,t}: \alpha \in A, t \in [0,T])$ is equicontinuous at zero.
\end{prop}

\begin{proof}[Proof of Theorem \ref{theoThighnessMeasures}] 
Let $(\mu_{\alpha}: \alpha \in A)$ be a family of probability measures on $D_{T}(\Phi'_{\beta})$ satisfying conditions  \ref{condEquiFouTRans} and \ref{condWeakTight} in  Theorem \ref{theoThighnessMeasures}. From Proposition \ref{propFamilyMeasRadonInCHS} 
there exists a weaker countably Hilbertian topology $\theta$ on $\Phi$ such that $\mu_{\alpha}$ is a Radon probability measure on $D_{T}((\widetilde{\Phi_{\theta}})'_{\beta})$ for each $\alpha \in A$.

Let $\epsilon >0$. Because $\widetilde{\Phi_{\theta}}$ is an ultrabornological space,  it follows from Proposition \ref{propSuffiCondTightOnUltrabor} that 
there exists a continuous Hilbertian seminorm $p$ on $\widetilde{\Phi_{\theta}}$ (therefore continuous on $\Phi$) such that 
$$ \sup_{\alpha \in A} \int_{D_{T}((\widetilde{\Phi_{\theta}})'_{\beta})} \, \sup_{t \in [0,T]} \abs{1-e^{ix(t)[\phi]}} d\mu_{\alpha} \leq \frac{\epsilon}{12}, \quad \forall \, \phi \in B_{p}(1). $$
Now, because $\abs{e^{ix(t)[\phi]}} \leq 1$ $\forall \, \phi \in \Phi$ and $x \in D_{T}((\widetilde{\Phi_{\theta}})'_{\beta})$, it then follows  that 
\begin{equation} \label{inequaSupCharaFunct}
\sup_{\alpha \in A} \int_{D_{T}((\widetilde{\Phi_{\theta}})'_{\beta})} \, \sup_{t \in [0,T]} \abs{1-e^{ix(t)[\phi]}} d\mu_{\alpha} \leq \frac{\epsilon}{12}+2p(\phi)^{2}, \quad \forall \, \phi \in \Phi. 
\end{equation}
Let $q$ be a continuous Hilbertian seminorm on $\Phi$ such that $p \leq q$ and $i_{p,q}$ is Hilbert-Schmidt. Let $(\phi^{q}_{k})_{k \in \N} \subseteq \Phi$ be a complete orthonormal system in $\Phi_{q}$.  By following similar arguments to those used in the proof of Lemma 3.8 in \cite{FonsecaMora:2018} (see also  Lemma 3.2 in \cite{Mitoma:1983}) and from \eqref{inequaSupCharaFunct} it follows that for every $C>0$ and every $\alpha \in A$:

$$  \mu_{\alpha} \left( x \in D_{T}((\widetilde{\Phi_{\theta}})'_{\beta}): \sup_{t \in [0,T]}  \sum_{k=1}^{\infty} \abs{x(t)[\phi_{k}^{q}]}^{2} > C^{2} \right) $$

$$ \leq  \lim_{m \rightarrow \infty} \frac{\sqrt{e}}{\sqrt{e}-1}  
  \int_{ D_{T}((\widetilde{\Phi_{\theta}})'_{\beta})} \sup_{t \in [0,T]}  \left( 1- \exp  \frac{-1}{2C^{2}} \sum_{k=1}^{m} \abs{x(t)[\phi_{k}^{q}]}^{2}   \right) d \mu_{\alpha} $$

$$ \leq  \lim_{m \rightarrow \infty} \frac{\sqrt{e}}{\sqrt{e}-1}  
    \int_{R^{m}}\int_{ D_{T}((\widetilde{\Phi_{\theta}})'_{\beta})} \sup_{t \in [0,T]}  \abs{ 1- \exp  i \sum_{k=1}^{m} \frac{z_{k} x(t)[\phi_{k}^{q}]}{2C^{2}}   } d \mu_{\alpha} \, \frac{e^{\frac{-\abs{z}^{2}}{2}}}{(2\pi)^{\frac{m}{2}}} dz $$

$$ \leq   \lim_{m \rightarrow \infty} \frac{\sqrt{e}}{\sqrt{e}-1} \left(\frac{\epsilon}{12} + \frac{2}{C^{2}} \sum_{k=1}^{m} p(\phi_{k}^{q})^{2}  \right) $$

$$ =  \frac{\sqrt{e}}{\sqrt{e}-1} \left(\frac{\epsilon}{12} + \frac{2}{C^{2}} \norm{ i_{p,q} }^{2}_{\mathcal{L}_{2}(\Phi_{q}, \Phi_{p}) }  \right). $$

Then choosing $C$ such that $\displaystyle{\frac{2}{C^{2}} \norm{ i_{p,q} }^{2}_{\mathcal{L}_{2}(\Phi_{q}, \Phi_{p}) } < \frac{\epsilon}{12}}$ and considering the probability of the complement, we get that 
\begin{multline} 
 \inf_{\alpha \in A} \mu_{\alpha} \left( x \in D_{T}(\Phi'_{q}): \sup_{t \in [0,T]} q'(x(t)) \leq C \right) \\
 \geq  \inf_{\alpha \in A} \mu_{\alpha} \left( x \in D_{T}((\widetilde{\Phi_{\theta}})'_{\beta}): \sup_{t \in [0,T]} \sum_{k=1}^{\infty} \abs{x(t)[\phi_{k}^{q}]}^{2} \leq C^{2} \right) 
\geq 1-\frac{\epsilon}{2}. \label{pointwiseBoundedOnHilbert}
\end{multline}
Let $\varrho$ be a continuous Hilbertian seminorm on $\Phi$ such that $q \leq \varrho$ and $i_{q,\varrho}$ is compact (e.g. Hilbert-Schmidt). Then $i'_{q,\varrho}$ is also  a compact operator. As $F=\{ f \in \Phi'_{q}: q'(f) \leq C \}$ is a neighborhood of zero in $\Phi'_{q}$, then its image under $i'_{q,\varrho}$ is a relatively compact subset of $\Phi'_{\varrho}$. Let $K$ be its closure. Then $K$ is a compact subset of $\Phi'_{\varrho}$, and regarding both $F$ and $K$ as subsets of $\Phi'_{\varrho}$ we have $F \subseteq K$. Then it follows from \eqref{pointwiseBoundedOnHilbert} that 
\begin{multline} 
\inf_{\alpha \in A} \mu_{\alpha} \left( x \in D_{T}(\Phi'_{\varrho}): x(t) \in K \, \mbox{ for all } t \in [0,T] \, \right) \\
\geq  \inf_{\alpha \in A} \mu_{\alpha} \left( x \in D_{T}(\Phi'_{q}): \sup_{t \in [0,T]} q'(x(t)) \leq C \right) \geq 1-\frac{\epsilon}{2}. \label{pointwiseTightnessEqua}
\end{multline}

Now, since $\Phi_{\varrho}$ is a separable Hilbert space, we can choose a sequence $(\varphi^{\varrho}_{j})_{j \in \N} \subseteq \Phi$ that separates points in $\Phi'_{\varrho}$ (see \cite{BogachevMT}, Proposition 6.5.4, p.17). For each $j \in \N$, from our assumption of tightness of $(\mu_{\alpha} \circ \Pi_{\varphi_{j}^{\varrho}}^{-1}: \alpha \in A)$, there exists a compact subset $B_{j}$ of $D_{T}(\R)$ such that 
\begin{equation} \label{measureSpaceProjectionsTightness}
\inf_{\alpha \in A} \mu_{\alpha} \circ \Pi_{\varphi_{j}^{\varrho}}^{-1} (B_{j})> 1-\frac{\epsilon}{2^{j+1}}.
\end{equation}  
Let 
$$\Gamma=\left( \bigcap_{j =1}^{\infty} \Pi_{\varphi_{j}^{\varrho}}^{-1} (B_{n}) \right) \cap \left\{ x \in D_{T}(\Phi'_{\varrho}): x(t) \in K \, \mbox{ for all } t \in [0,T] \, \right\}.$$
The set $\Gamma$ satisfies: 
\begin{enumerate}
\item $\Gamma \subseteq D_{T}(K)$ is closed; where $K \subseteq \Phi'_{\varrho}$ is equipped with the subspace topology,
\item for each $j \in \N$, $\Pi_{\varphi_{j}^{\varrho}}(\Gamma)$ is compact in $D_{T}(\R)$. 
\end{enumerate}
Then, Lemma 3.3 in \cite{Jakubowski:1986} shows that 
$\Gamma$ is compact in  $D_{T}(\Phi'_{\varrho})$. Moreover, from \eqref{pointwiseTightnessEqua} and \eqref{measureSpaceProjectionsTightness}
$$\sup_{\alpha \in A} \mu_{\alpha} (\Gamma^{c}) \leq \frac{\epsilon}{2}+ \sum_{j=1}^{\infty} \frac{\epsilon}{2^{j+1}}=\epsilon. $$
 
Now consider a decreasing sequence of positive real numbers $(\epsilon_{n}: n \in \N)$ converging to zero, and for each $n \in \N$ choose a continuous Hilbertian seminorm $\varrho_{n}$ on $\Phi$ and a compact subset $\Gamma_{n}$ in $D_{T}(\Phi'_{\varrho_{n}})$ such that 
$$ \inf_{\alpha \in A} \mu_{\alpha} (\Gamma_{n}) \geq 1-\epsilon_{n},$$
then it is clear from \eqref{decompSkorSpaceWeakCountHilb} that $(\mu_{\alpha}: \alpha \in A)$ is uniformly tight on $D_{T}((\widetilde{\Phi_{\vartheta}})'_{\beta})$, where $\vartheta$ is the weaker countably Hilbertian topology on $\Phi$ generated by the seminorms $(\varrho_{n})$. Here it is important to stress the fact that by construction the topology $\vartheta$ is finner than the topology $\theta$ defined at the beginning of the proof and hence each $\mu_{\alpha}$ is a Radon measure on $D_{T}((\widetilde{\Phi_{\vartheta}})'_{\beta})$. Now, because the inclusion from $D_{T}((\widetilde{\Phi_{\vartheta}})'_{\beta})$ into $D_{T}(\Phi'_{\beta})$ is continuous (Proposition \ref{propProperSkoroSpaceWeakerCHT}(3)),  then we have that $(\mu_{\alpha}: \alpha \in A)$  is uniformly tight on $D_{T}(\Phi'_{\beta})$.

To prove the converse, assume that $\Phi$ is a barrelled nuclear space and that the family $(\mu_{\alpha}: \alpha \in A)$ is uniformly tight on $D_{T}(\Phi'_{\beta})$. Then it follows from Proposition 1.6.vi) in \cite{Jakubowski:1986} that the family of probability measures $(\mu_{\alpha}\circ \Pi^{-1}_{t}: \alpha \in A, t \in [0,T])$ is uniformly tight on $\Phi'_{\beta}$. But because $\Phi$ is barrelled and nuclear, it follows that the family of its Fourier transforms $(\widehat{\mu}_{\alpha,t}: t \in [0,T], \alpha \in A)$ is equicontinuous at zero (see \cite{DaleckyFomin}, Theorem III.2.7, p.104). Finally for each $\phi \in \Phi$ the continuity of the space projection map $\Pi_{\phi}$ and that  $(\mu_{\alpha}: \alpha \in A)$ is uniformly tight on $D_{T}(\Phi'_{\beta})$ implies that  $(\mu_{\alpha} \circ \Pi_{\phi}^{-1}: \alpha \in A)$ is uniformly tight on $D_{T}(\R)$.    
\end{proof}

Now, observe that if $\Phi$ is an ultrabornological nuclear space then it follows from  
Proposition \ref{propSuffiCondTightOnUltrabor} that the condition \ref{condWeakTight} in Theorem \ref{theoThighnessMeasures} implies condition \ref{condEquiFouTRans} in that theorem. From this observation we obtain the following important result.

\begin{theo} \label{theoTightMeasuresUltrab}
Let $(\mu_{\alpha}: \alpha \in A)$ be a family of probability measures on $D_{T}(\Phi'_{\beta})$ where $\Phi$ is an ultrabornological nuclear space. Then, the family $(\mu_{\alpha}: \alpha \in A)$ is uniformly tight on $D_{T}(\Phi'_{\beta})$ if and only if $\forall \phi \in \Phi$ the family $(\mu_{\alpha} \circ \Pi_{\phi}^{-1}: \alpha \in A)$ is uniformly tight on $D_{T}(\R)$. Moreover, there exists a weaker countably Hilbertian topology $\theta$ on $\Phi$ such that $(\mu_{\alpha}: \alpha \in A)$ is uniformly tight on $D_{T}((\widetilde{\Phi_{\theta}})'_{\beta})$.  
\end{theo}

We finalize this section with the proof of Proposition \ref{propSuffiCondTightOnUltrabor}. To do this, we will need the following preliminary results on pseudo-seminorms on vector spaces. 

\begin{defi}
A function $x \mapsto \abs{x}$ defined on a vector space $L$ over $\R$ is called a \emph{pseudo-seminorm} on $L$ if 
\begin{enumerate}
\item $\abs{x+y} \leq \abs{x}+ \abs{y}$, $\forall \, x, y \in L$. 
\item For $\lambda \in \R$, $\abs{\lambda} \leq 1$ implies $\abs{ \lambda x} \leq \abs{x}$, $\forall \, x \in L$.
\item If $\lambda_{n} \rightarrow 0$, then $\abs{\lambda_{n} x} \rightarrow 0$, $\forall \, x \in L$. 
\item $\abs{x_{n}} \rightarrow 0$ implies $\abs{\lambda x_{n}} \rightarrow 0$, $\forall \lambda \in \R$. 
\end{enumerate}
\end{defi}

Every pseudo-seminorm defines a pseudo-metric $d(x,y)=\abs{x-y}$ on $L$ that generates a metrizable linear topology on $L$ (see \cite{Schaefer}, Section I.6).

The following facts are relevant to our study of pseudo-seminorms on ultrabornological spaces:

\begin{enumerate}
\item A topological vector space $E$ is \emph{sequential} if every sequentially closed subset is closed. Hence every sequentially lower semicontinuous function on $E$ is lower semicontinuous. 
\item $E$ is sequential if and only if every sequentially continuous pseudo-seminorm on $E$ is continuous if and only if every sequentially continuous linear map to an arbitrary Fr\'{e}chet space is continuous (see \cite{KatsarasBenekas:1995}, Proposition 2.6). Hence every bornological space is sequential (\cite{KotheI}, Theorem 28.3.(4), p.383)
\item A topological vector space $E$ is \emph{S-barrelled} if and only if every lower semicontinuous pseudo seminorm on $E$ is sequentially continuous (see \cite{KatsarasBenekas:1995}, Proposition 5.2) if and only if every pointwise bounded family of continuous linear maps from $E$ into an arbitrary Fr\'{e}chet space is sequentially equicontinuous (see \cite{KatsarasBenekas:1995}, Proposition 5.5). 
\end{enumerate}

From facts (1)-(3) above and because every ultrabornological space is bornological and barrelled, we have the following:

\begin{prop}\label{propUltraborAndPseudoSeminorm}
If $E$ is an ultrabornological space, then every sequentially lower semicontinuous pseudo-seminorm on $E$ is continuous. 
\end{prop}

The main step in the proof of Proposition \ref{propSuffiCondTightOnUltrabor} is given in the following result. 

\begin{lemm}\label{lemmPseudoSeminormMeasures}
Let $(\mu_{\alpha}: \alpha \in A)$ be a family of probability measures on $D_{T}(\Psi_{\beta})$ where $\Psi$ is an ultrabornological space, and suppose that $\forall \phi \in \Psi$ the family $(\mu_{\alpha} \circ \Pi_{\phi}^{-1}: \alpha \in A)$ is uniformly tight on $D_{T}(\R)$. Let $V: \Psi \rightarrow [0,+\infty)$ be given by 
$$V(\phi)=\sup_{\alpha \in A} \int_{D_{T}(\Psi'_{\beta})} \frac{\sup_{t \in [0,T]} \abs{x(t)[\phi]}}{1+\sup_{t \in [0,T]} \abs{x(t)[\phi]}} d \mu_{\alpha}, \quad \forall \, \phi \in \Psi. $$
Then, $V$ is a continuous pseudo-seminorm on $\Psi$. 
\end{lemm}
\begin{prf}
We show first that $V$ is a pseudo-seminorm. 
\begin{enumerate}
\item If $\phi_{1}, \phi_{2} \in \Psi$, then because $x \mapsto \frac{x}{1+x}$ is a subadditive function it is clear that $V(\phi_{1}+\phi_{2}) \leq V(\phi_{1})+V(\phi_{2})$. 
\item Let $\phi \in \Psi$, $\lambda \in \R$, $\abs{\lambda} \leq 1$. Because the function $x \mapsto \frac{x}{1+x}$ is increasing, $V(\lambda \phi) \leq V(\phi)$. 
\item We follow some ideas from the proof Lemma 3.3 in \cite{Mitoma:1983}. Let $\lambda_{m} \rightarrow 0$, $\epsilon>0$ , and $\phi \in \Psi$. Because the family $(\mu_{\alpha} \circ \Pi_{\phi}^{-1}: \alpha \in A)$ is uniformly tight on $D_{T}(\R)$, it follows from Theorem 2.4.3 in \cite{KallianpurXiong} that there exists $r(\epsilon)$ such that for all $\alpha \in A$
\begin{flalign*}
& \mu_{\alpha} \left( x \in D_{T}(\Psi'_{\beta}): \sup_{t \in [0,T]} \abs{x(t)[\phi]} > r(\epsilon) \right) \\
& = \mu_{\alpha} \circ \Pi_{\phi}^{-1} \left( y \in D_{T}(\R): \sup_{t \in [0,T]} \abs{y(t)} > r(\epsilon) \right)< \epsilon. 
\end{flalign*}
Let $N \in \N$ such that $\forall m \geq N$, $\lambda_{m} r(\epsilon) < \epsilon$. Let 
$$\Gamma=\left\{ x \in D_{T}(\Psi'_{\beta}): \sup_{t \in [0,T]} \abs{x(t)[\phi]} > r(\epsilon) \right\}.$$
Then, $\forall m \geq N$ we have 
\begin{eqnarray*}
V(\lambda_{m}\phi) & = & \sup_{\alpha \in A} \int_{D_{T}(\Psi'_{\beta})} \frac{\sup_{t \in [0,T]} \abs{x(t)[\lambda_{m}\phi]}}{1+\sup_{t \in [0,T]} \abs{x(t)[\lambda_{m}\phi]}} d \mu_{\alpha}  \\ 
& \leq &  \sup_{\alpha \in A}  \mu_{\alpha}(\Gamma) + \frac{\epsilon}{1+\epsilon} < 2 \epsilon.
\end{eqnarray*}
Therefore, $\displaystyle{\lim_{m \rightarrow \infty} V(\lambda_{m} \phi)=0}$. 
\item Let $\lambda \in \R$ and let $(\phi_{m} )_{m \in \N} \subseteq \Psi$ such that $\displaystyle{\lim_{m \rightarrow \infty} V(\phi_{m})=0}$. Consider any subsequence $(\phi_{m_{k}} )_{k \in \N}$ of $(\phi_{m} )_{m \in \N}$. Then  $\displaystyle{\lim_{k \rightarrow \infty} V(\phi_{m_{k}})=0}$. Hence, for each $r \in \N$ there exists $\phi_{m_{k(r)}}$ such that 
\begin{equation*}
V(\phi_{m_{k(r)}})=\sup_{\alpha \in A} \int_{D_{T}(\Psi'_{\beta})} \frac{\sup_{t \in [0,T]} \abs{x(t)[\phi_{m_{k(r)}}]}}{1+\sup_{t \in [0,T]} \abs{x(t)[\phi_{m_{k(r)}}]}} d \mu_{\alpha} \leq \frac{1}{2^{r+2}}.  
\end{equation*} 
Therefore, 
\begin{flalign*}
& \sup_{\alpha \in A}  \mu_{\alpha} \left( x \in D_{T}(\Psi'_{\beta}):  \sup_{t \in [0,T]} \abs{x(t)[\phi_{m_{k(r)}}]} > 2^{-r} \right) \\
& = \sup_{\alpha \in A}  \mu_{\alpha} \left( x \in D_{T}(\Psi'_{\beta}):  \frac{\sup_{t \in [0,T]} \abs{x(t)[\phi_{m_{k(r)}}]}}{1+\sup_{t \in [0,T]} \abs{x(t)[\phi_{m_{k(r)}}]}} > \frac{2^{-r}}{1+2^{-r}} \right) \\
& \leq \frac{1+2^{-r}}{2^{-r}} \sup_{\alpha \in A} \int_{D_{T}(\Psi'_{\beta})} \frac{\sup_{t \in [0,T]} \abs{x(t)[\phi_{m_{k(r)}}]}}{1+\sup_{t \in [0,T]} \abs{x(t)[\phi_{m_{k(r)}}]}} d \mu_{\alpha} \\
& \leq 3 \cdot \frac{1}{2^{r+2}} \leq \frac{1}{2^{r}} 
\end{flalign*}
Then, it follows that for every $r \in \N$ we have
\begin{eqnarray*}
V(\lambda \phi_{m_{k(r)}}) 
& \leq & \sup_{\alpha \in A}  \mu_{\alpha} \left( x \in D_{T}(\Psi'_{\beta}):  \sup_{t \in [0,T]} \abs{x(t)[\lambda \phi_{m_{k(r)}}]} > \abs{\lambda} 2^{-r} \right) \\
& {} & + \frac{\abs{\lambda}2^{-r}}{1+\abs{\lambda}2^{-r}} \\
& < & 2^{-r} ( 1+\abs{\lambda}).
\end{eqnarray*}
So, we conclude that $\displaystyle{\lim_{r \rightarrow \infty} V(\lambda \phi_{m_{k(r)}})=0}$. Then, as each subsequence of $(V(\lambda \phi_{m}): m \in \N)$ has a further subsequence that converges to $0$, it follows that $\displaystyle{\lim_{m \rightarrow \infty} V(\lambda \phi_{m})=0}$. 
\end{enumerate}

Thus we have shown that $V$ is a pseudo-seminorm on $\Psi$. Our next objective is to show that $V$ is sequentially lower semicontinuous. 

Let $(\phi_{m}: m \in \N)$ be a sequence in $\Psi$ converging to $\phi \in \Psi$. Because for each $x \in D_{T}(\Psi'_{\beta})$ the map $\varphi \mapsto \sup_{t \in [0,T]} \abs{x(t)[\varphi]}$ is lower semicontinuous, it follows from Fatou's lemma that 
\begin{eqnarray*}
V(\phi) 
& \leq & \sup_{\alpha \in A} \int_{D_{T}(\Psi'_{\beta})} \liminf_{m \rightarrow \infty} \frac{\sup_{t \in [0,T]} \abs{x(t)[\phi_{m}]}}{1+\sup_{t \in [0,T]} \abs{x(t)[\phi_{m}]}} d \mu_{\alpha} \\
& \leq & \sup_{\alpha \in A} \liminf_{m \rightarrow \infty} \int_{D_{T}(\Psi'_{\beta})}  \frac{\sup_{t \in [0,T]} \abs{x(t)[\phi_{m}]}}{1+\sup_{t \in [0,T]} \abs{x(t)[\phi_{m}]}} d \mu_{\alpha} \\
& \leq & \liminf_{m \rightarrow \infty} V(\phi_{m}). 
\end{eqnarray*}
Hence $V$ is a sequentially lower semicontinuous pseudo-seminorm on $\Psi$ and  because this space is ultrabornological, Proposition \ref{propUltraborAndPseudoSeminorm} shows that $V$ is continuous on $\Psi$. 
\end{prf}

\begin{proof}[Proof of Proposition  \ref{propSuffiCondTightOnUltrabor}]
Let $\epsilon >0$. From the continuity of the exponential function there exists $\delta_{1} >0$ such that $\abs{ 1-e^{ir}} \leq \frac{\epsilon}{2}$ whenever $\abs{x} < \delta_{1}$. Now, by Lemma \ref{lemmPseudoSeminormMeasures} there exists a continuous seminorm $p$ on $\Psi$ such that $V(\phi) \leq (\delta_{2})^{2}$ $\forall \phi \in B_{p}(1)$, where $\delta_{2} = \min\left\{\delta_{1}, \frac{-1+\sqrt{1+\epsilon}}{2} \right\}$. For given $\phi \in \Psi$, let $\Gamma_{\phi} =\{  x \in D_{T}(\Psi'_{\beta}):  \sup_{t \in [0,T]} \abs{x(t)[\phi]} \leq \delta_{2} \}$. 

Then, $\forall \phi \in B_{p}(1)$ we have that 
\begin{eqnarray*}
\sup_{\alpha \in A} \int_{D_{T}(\Psi'_{\beta})} \, \sup_{t \in [0,T]} \abs{1-e^{ix(t)[\phi]}} d\mu_{\alpha} 
& \leq & \sup_{\alpha \in A} \int_{\Gamma_{\phi}} \, \sup_{t \in [0,T]} \abs{1-e^{ix(t)[\phi]}} d\mu_{\alpha} \\
& {} & + 2 \sup_{\alpha \in A} \mu_{\alpha} (\Gamma_{\phi}^{c}) \\
& \leq & \frac{\epsilon}{2} +2 \frac{(1+\delta_{2})}{\delta_{2}} V(\phi)\\
& \leq & \frac{\epsilon}{2} +2 \frac{\epsilon}{4}= \epsilon. 
\end{eqnarray*}
Finally, the equicontinuity at zero of the family of Fourier transforms $(\widehat{\mu}_{\alpha,t}: \alpha \in A, t \in [0,T])$ is just a consequence of the above result and  the inequality: 
$$ \sup_{\alpha \in A} \sup_{t \in [0,T]} \abs{1-\widehat{\mu}_{\alpha,t}(\phi)}  \leq \sup_{\alpha \in A} \int_{D_{T}(\Psi'_{\beta})} \, \sup_{t \in [0,T]} \abs{1-e^{ix(t)[\phi]}} d\mu_{\alpha}, \quad \forall \, \phi \in \Psi. $$ 
\end{proof}

\subsection{Uniform tightness on $D_{\infty}(\Phi'_{\beta})$} \label{sectionUNSPINFTY}

Let $D_{\infty}(\Phi'_{\beta})$  denote the space of mappings $x: [0,\infty) \rightarrow \Phi'_{\beta}$ wich are c\`{a}dl\`{a}g. For every $s \geq 0$, let $r_{s}: D(\Phi'_{\beta}) \rightarrow D_{s+1}(\Phi'_{\beta})$  be given by
$$r_{s}(x)(t) =
\begin{cases}
x(t) & \mbox{ if } t \in [0,s],\\
(s+1-t)x(t) & \mbox{ if } t \in [s,s+1].
\end{cases}
$$
For every $\gamma \in \Gamma$ let
$$d^{\infty}_{\gamma}(x,y)=\sum_{n =1}^{\infty} \frac{1}{2^{n}} \left( 1 \wedge d^{n}_{\gamma}(r_{n}(x),r_{n}(y)) \right),$$
where for each $n \in \N$, $d^{n}_{\gamma}$ is the pseudometric defined in  \eqref{defSkorokhodPseudometrics} for $T=n$. It is not hard to check that each $d^{\infty}_{\gamma}$ is a pseudometric in $D_{\infty}(\Phi'_{\beta})$. The \emph{Skorokhod topology} in $D_{\infty}(\Phi'_{\beta})$ (see \cite{Jakubowski:1986, Mitoma:1983}) is the completely regular topology generated by the family $( d^{\infty}_{\gamma}: \gamma \in \Gamma)$. An equivalent topology is obtained if we consider a family of seminorms other than $(q_{\gamma}: \gamma \in \Gamma)$ that generates the strong topology $\beta$ on $\Phi'$ (see \cite{Jakubowski:1986}, Theorem 4.3).

An interesting fact on the topology introduced above is that if for some $T>0$ we have that $\mathcal{C}(D_{T}(\Phi'_{\beta})) =  \mathcal{B}(D_{T}(\Phi'_{\beta}))$,  we also have $\mathcal{C}(D_{\infty}(\Phi'_{\beta})) =  \mathcal{B}(D_{\infty}(\Phi'_{\beta}))$ (see Proposition 4.4 in \cite{Jakubowski:1986} and Lemma 9 in \cite{Kouritzin:2016}). Hence, if $\theta$ is a weaker countably Hilbertian topology on  $\Phi$, it follows from Lemma \ref{lemmCylinAndBorelSigmaAlgCoincide} that  $\mathcal{C}(D_{\infty}((\widetilde{\Phi_{\theta}})'_{\beta})) =  \mathcal{B}(D_{\infty}((\widetilde{\Phi_{\theta}})'_{\beta}))$. Moreover, because for each $T>0$ the compact subsets of $D_{T}((\widetilde{\Phi_{\theta}})'_{\beta})$ are metrizable (see Proposition \ref{propProperSkoroSpaceWeakerCHT}(2)), and the canonical inclusion from $D_{\infty}((\widetilde{\Phi_{\theta}})'_{\beta})$ into $D_{T}((\widetilde{\Phi_{\theta}})'_{\beta})$ is continuous, then the compact subsets of $D_{\infty}((\widetilde{\Phi_{\theta}})'_{\beta})$ are also metrizable. 

It is clear from the above arguments that all the concepts of measures and random elements introduced in Sect. \ref{sectionMeRVSkoSpac} and the results proved there are also valid for the space $D_{\infty}(\Phi'_{\beta})$. We leave the reader the task to complete the details. 

The following theorem provides necessary and sufficient conditions for uniform tightness for probability measures on $D_{\infty}(\Phi'_{\beta})$. 

\begin{theo}\label{theoUniformTightnessInfty} 
Let $(\mu_{\alpha}: \alpha \in A)$ be a family of probability measures on $D_{\infty}(\Phi'_{\beta})$ such that it satisfies the following:
\begin{enumerate}
\item \label{condEquiFouTRansInfty} For all $T>0$, the family of Fourier transforms $(\widehat{\mu}_{\alpha,t}: t \in [0,T], \alpha \in A)$ is equicontinuous at zero.
\item \label{condWeakTightInfty} For each  $\phi \in \Phi$, the family $(\mu_{\alpha} \circ \Pi_{\phi}^{-1}: \alpha \in A)$ of probability measures on $D_{\infty}(\R)$ is uniformly tight.    
\end{enumerate}  
Then there exists a weaker countably Hilbertian topology $\theta$ on $\Phi$ such that $(\mu_{\alpha}: \alpha \in A)$ is uniformly tight on $D_{\infty}((\widetilde{\Phi_{\theta}})'_{\beta})$. In particular, the family $(\mu_{\alpha}: \alpha \in A)$ is uniformly tight on $D_{\infty}(\Phi'_{\beta})$. 

Conversely, if $\Phi$ is a barrelled nuclear space and the family $(\mu_{\alpha}: \alpha \in A)$ is uniformly tight on $D_{\infty}(\Phi'_{\beta})$, then conditions \ref{condEquiFouTRans} and \ref{condWeakTight} are satisfied. 
\end{theo}
\begin{prf}
We start by showing the following: 

\textbf{Claim:} given $\epsilon>0$, there exists a weaker countably Hilbertian topology $\theta_{\epsilon}$ on $\Phi$ and a compact $\mathcal{K}_{\epsilon}$ on $D_{\infty}((\widetilde{\Phi_{\theta_{\epsilon}}})'_{\beta})$ such that $\sup_{\alpha \in A} \mu_{\alpha}(\mathcal{K}_{\epsilon}^{c})  < \epsilon$.

Let $\epsilon >0$. From Theorem \ref{theoThighnessMeasures}, for each $n \in \N$ there exists a continuous Hilbertian seminorm $q_{n}$ on $\Phi$ and $\mathcal{K}_{n} \subseteq D_{n}(\Phi'_{q_{n}})$ compact such that $\sup_{\alpha \in A} \mu_{\alpha} \circ r_{n}^{-1} (\mathcal{K}_{n}^{c}) < \epsilon / 2^{n}$. Then, it follows from Proposition 1.6.iv) in \cite{Jakubowski:1986} that for each $n \in \N$ there exists a compact $K_{n} \subseteq \Phi'_{q_{n}}$ such that $\mathcal{K}_{n} \subseteq D_{n}(K_{n})$. 

Let $\mathcal{K}= \bigcap_{n=1}^{\infty} r_{n}^{-1}(\mathcal{K}_{n}) \subseteq D_{\infty}((\widetilde{\Phi_{\theta}})'_{\beta})$, where $\theta$ is the weaker countably Hilbertian topology on $\Phi$ generated by the family $(q_{n}: n \in \N)$. Because for all $n \in \N$, $r_{n}^{-1}(\mathcal{K}_{n}) \subseteq D_{\infty}(K_{n})$, then if we take $K = \bigcap_{n=1}^{\infty} K_{n}$, then $K$ is a compact subset in $(\widetilde{\Phi_{\theta}})'_{\beta}$ and moreover $\mathcal{K} \subseteq D_{\infty}(K)$ is closed. Hence because $\widetilde{\Phi_{\theta}}$ is separable and metrizable, and $\forall \phi \in \Phi$ we have $\Pi_{\phi}(\mathcal{K})$ is compact in $D_{\infty}(\R)$, then Lemma 3.3 in \cite{Jakubowski:1986} shows that $\mathcal{K}$ is compact in $D_{\infty}((\widetilde{\Phi_{\theta}})'_{\beta})$. Moreover observe that 
$$ \sup_{\alpha \in A} \mu_{\alpha}(\mathcal{K}^{c}) \leq \sup_{\alpha \in A} \sum_{n=1}^{\infty} \mu_{\alpha} \circ r_{n}^{-1} (\mathcal{K}_{n}^{c}) < \epsilon.$$
So we have proved our claim. 

Now, if $(\epsilon_{m}: m \in \N)$ is a decreasing sequence of positive real numbers converging to $0$, then for each $m \in \N$ there exists $\theta_{m}$ and $\mathcal{K}_{m}$ satisfying the properties stated on the claim. But then, if $\theta$ is the weaker countably topology on $\Phi$ generated by the Hilbertian seminorms generating the topologies $\theta_{m}$ for $m \in \N$, then each $\mathcal{K}_{m}$ is compact in $D_{\infty}((\widetilde{\Phi_{\theta}})'_{\beta})$ and the family $(\mu_{\alpha}: \alpha \in A)$ is therefore uniformly tight on $D_{\infty}((\widetilde{\Phi_{\theta}})'_{\beta})$, and hence is uniformly tight on $D_{\infty}(\Phi'_{\beta})$.   

For the converse, if $\Phi$ is barrelled and $(\mu_{\alpha}: \alpha \in A)$ is tight on $D_{\infty}(\Phi'_{\beta})$, then for each $T>0$ the family $(\mu_{\alpha} \circ r_{T}^{-1} : \alpha \in A)$ is tight on $D_{T}((\widetilde{\Phi_{\theta}})'_{\beta})$ and the result follows from Theorem \ref{theoThighnessMeasures}.    
\end{prf}

If in the above proof we use Theorem \ref{theoTightMeasuresUltrab} instead of Theorem \ref{theoThighnessMeasures}, we get the following result for ultrabornological nuclear spaces.

\begin{theo} \label{theoUniformTightnessUltrabInfty}
Let $(\mu_{\alpha}: \alpha \in A)$ be a family of probability measures on $D_{\infty}(\Phi'_{\beta})$ where $\Phi$ is an ultrabornological nuclear space. Then the family $(\mu_{\alpha}: \alpha \in A)$ is uniformly tight on $D_{\infty}(\Phi'_{\beta})$ if and only if $\forall \phi \in \Phi$ the family $(\mu_{\alpha} \circ \Pi_{\phi}^{-1}: \alpha \in A)$ is uniformly tight on $D_{\infty}(\R)$. Moreover there exists a weaker countably Hilbertian topology $\theta$ on $\Phi$ such that $(\mu_{\alpha}: \alpha \in A)$ is uniformly tight on $D_{\infty}((\widetilde{\Phi_{\theta}})'_{\beta})$.  
\end{theo}

\section{Weak Convergence on the Skorokhod Space} \label{sectionWeaConvSkoSpa}

\begin{assu}
Unless otherwise indicated, in this section we will always assume that $\Phi$ is a nuclear space. 
\end{assu}

\subsection{Weak Convergence of Probability Measures} 

The following result shows how the theory introduced on the previous sections can be used to prove the weak convergence of a sequence of probability measures on $D_{\infty}(\Phi'_{\beta})$. 

\begin{theo}[L\'{e}vy's continuity theorem on $D_{\infty}(\Phi'_{\beta})$] \label{theoWeakConveMeasures}
Let $(\mu_{n}: n \in \N)$ be a sequence of probability measures on $D_{\infty}(\Phi'_{\beta})$ that satisfies the following:
\begin{enumerate}
\item For each $T>0$, the family of Fourier transforms $(\widehat{\mu}_{n,t}: t \in [0,T], n \in \N)$ is equicontinuous at zero.
\item For each $\phi \in \Phi$, the family $(\mu_{n} \circ \Pi_{\phi}^{-1}: n \in \N)$  of probability measures on $D_{\infty}(\R)$ is uniformly tight.
\item $\forall$ $m \in \N$, $\phi_{1}, \dots, \phi_{m} \in \Phi$, $t_{1}, \dots, t_{m} \geq 0$, 
$$ \mu_{n} \circ \left( \Pi^{\phi_{1}, \dots, \phi_{m}}_{t_{1}, \dots, t_{m}} \right)^{-1} \Rightarrow \nu^{\phi_{1}, \dots, \phi_{m}}_{t_{1}, \dots, t_{m}}, $$    
where $\nu^{\phi_{1}, \dots, \phi_{m}}_{t_{1}, \dots, t_{m}}$ is a Borel probability measure on $\R^{m}$. 
\end{enumerate}
Then there exist a weaker countably Hilbertian topology $\theta$ on $\Phi$ and a  probability measure $\mu$ on $D_{\infty}((\widetilde{\Phi_{\theta}})'_{\beta})$ such that $\mu_{n} \Rightarrow \mu$ in $\goth{M}^{1}(D_{\infty}((\widetilde{\Phi_{\theta}})'_{\beta}))$. Moreover $\mu$ is the unique (up to equivalence) probability measure on $D_{\infty}(\Phi'_{\beta})$ such that $\mu_{n} \Rightarrow \mu$ in $\goth{M}^{1}(D_{\infty}(\Phi'_{\beta}))$. 
\end{theo}
\begin{prf}
First by \emph{(1)}, \emph{(2)} and Theorem \ref{theoUniformTightnessInfty}, there exists a weaker countably Hilbertian topology $\theta$ on $\Phi$ such that $(\mu_{n}: n \in \N)$ is uniformly tight on $D_{\infty}((\widetilde{\Phi_{\theta}})'_{\beta})$. 
As $D_{\infty}((\widetilde{\Phi_{\theta}})'_{\beta})$ is a completely regular topological  space whose compact subsets are metrizable (see Sect. \ref{sectionUNSPINFTY}), the fact that $(\mu_{n}: n \in \N)$ is uniformly tight on $D_{\infty}((\widetilde{\Phi_{\theta}})'_{\beta})$ implies that every subsequence of $(\mu_{n}: n \in \N)$ contains a further weakly convergent subsequence (see \cite{BogachevMT}, Theorem 8.6.7, p.206).
    
Let $(\mu^{1}_{n}: n \in \N)$ and $(\mu^{2}_{n}: n \in \N)$ two subsequences of $(\mu_{n}: n \in \N)$. Then, $(\mu^{1}_{n}: n \in \N)$ has a subsequence $(\mu^{1}_{n_{k}}: k \in \N)$ that converges weakly to $\nu^{1}$ and $(\mu^{2}_{n}: n \in \N)$ has a subsequence $(\mu^{2}_{n_{k}}: k \in \N)$ that converges weakly to $\nu^{2}$. The hypothesis \emph{(3)} shows that 
$$ \nu_{1} \circ \left( \Pi^{\phi_{1}, \dots, \phi_{m}}_{t_{1}, \dots, t_{m}} \right)^{-1} = \nu_{2} \circ \left( \Pi^{\phi_{1}, \dots, \phi_{m}}_{t_{1}, \dots, t_{m}} \right)^{-1}, $$  
for every $m \in \N$, $\phi_{1}, \dots, \phi_{m} \in \Phi$ and $t_{1}, \dots, t_{m} \ge 0$. 
Therefore, $\nu_{1}$ and $\nu_{2}$ coincide on all the cylinder sets and hence they coincide on $\mathcal{C}(D_{\infty}((\widetilde{\Phi_{\theta}})'_{\beta}))$. But because $\mathcal{C}(D_{\infty}((\widetilde{\Phi_{\theta}})'_{\beta}))=
\mathcal{B}(D_{\infty}((\widetilde{\Phi_{\theta}})'_{\beta}))$ (see Sect. \ref{sectionUNSPINFTY}), then $\nu_{1}=\nu_{2}$.
We have shown that every subsequence of $(\mu_{n}: n \in \N)$ contains a further subsequence that converges weakly to the same limit $\mu$ in $\goth{M}^{1}(D_{T}((\widetilde{\Phi_{\theta}})'_{\beta}))$. Hence, Theorem 2.6 in \cite{Billingsley} shows that $\mu_{n} \Rightarrow \mu$ in $\goth{M}^{1}((\widetilde{\Phi_{\theta}})'_{\beta})$. 

Finally because the inclusion $j_{\theta}$ from $(\widetilde{\Phi_{\theta}})'_{\beta}$ into $\Phi'_{\beta}$ is linear and continuous, then $\forall \, f \in C_{b}(\Phi'_{\beta})$ we have $f \circ j_{\theta} \in C_{b}(\goth{M}^{1}((\widetilde{\Phi_{\theta}})'_{\beta}))$. Therefore, the fact that $\mu_{n} \Rightarrow \mu$ in $\goth{M}^{1}(D_{\infty}((\widetilde{\Phi_{\theta}})'_{\beta}))$ implies that $\mu_{n} \Rightarrow \mu$ in $\goth{M}^{1}(D_{\infty}(\Phi'_{\beta}))$. This also shows the uniqueness of $\mu$. 
\end{prf}

For the case when $\Phi$ is an ultrabornological nuclear space, if in the proof of Theorem \ref{theoWeakConveMeasures} we use Theorem \ref{theoUniformTightnessUltrabInfty} instead of Theorem \ref{theoUniformTightnessInfty} we obtain the following result. 
 
\begin{theo} \label{theoWeakConveMeasuresUltrab}
Let $\Phi$ be an ultrabornological nuclear space and let $(\mu_{n}: n \in \N)$ be a sequence of Borel probability measures on $D_{\infty}(\Phi'_{\beta})$ such that for each $\phi \in \Phi$, the family $(\mu_{n} \circ \Pi_{\phi}^{-1}: n \in \N)$ is  tight on $D_{\infty}(\R)$.  Assume further that $\forall$ $m \in \N$, $\phi_{1}, \dots, \phi_{m} \in \Phi$, $t_{1}, \dots, t_{m} \geq 0$, there exists a probability measure $\nu^{\phi_{1}, \dots, \phi_{m}}_{t_{1}, \dots, t_{m}}$ on $\R^{m}$ such that
$$ \mu_{n} \circ \left( \Pi^{\phi_{1}, \dots, \phi_{m}}_{t_{1}, \dots, t_{m}} \right)^{-1} \Rightarrow \nu^{\phi_{1}, \dots, \phi_{m}}_{t_{1}, \dots, t_{m}}. $$    
Then there exist a weaker countably Hilbertian topology $\theta$ on $\Phi$ and a  probability measure $\mu$ on $D_{\infty}((\widetilde{\Phi_{\theta}})'_{\beta})$ such that $\mu_{n} \Rightarrow \mu$ in $\goth{M}^{1}(D_{\infty}((\widetilde{\Phi_{\theta}})'_{\beta}))$. Moreover $\mu$ is the unique (up to equivalence) probability measure on $D_{\infty}(\Phi'_{\beta})$ such that $\mu_{n} \Rightarrow \mu$ in $\goth{M}^{1}(D_{\infty}(\Phi'_{\beta}))$. 
\end{theo}

\begin{rema}
Clearly, Theorems \ref{theoWeakConveMeasures} and \ref{theoWeakConveMeasuresUltrab} can be also formulated for measures on $D_{T}(\Phi'_{\beta})$. We leave to the reader  the task of stating and proving them by using Theorems \ref{theoThighnessMeasures} and \ref{theoTightMeasuresUltrab}. 
\end{rema}

\subsection{Weak Convergence of (Cylindrical) Processes in the Skorokhod Space}

In this section we apply our results to provide sufficient conditions for the weak convergence in $D_{\infty}(\Phi'_{\beta})$ of a sequence of c\`{a}dl\`{a}g processes. This is done in the following result formulated in the more general setting of cylindrical processes:

\begin{theo}  \label{theoWeakConvProcesses}
For each $n \in \N$, let $X^{n}=\{X^{n}_{t} \}_{t \geq 0}$ be a cylindrical process in $\Phi'$ (e.g. a $\Phi'_{\beta}$-valued process) such that: 
\begin{enumerate}
\item For each $\phi \in \Phi$ and each $n \in \N$, the real-valued process $X^{n}(\phi)=\{ X^{n}_{t}(\phi) \}_{t \geq 0}$ is c\`{a}dl\`{a}g.
\item For every $T>0$, the family $\{ X^{n}_{t}: t \in [0,T], n \in \N \}$ of linear maps from $\Phi$ into $L^{0} \ProbSpace$ is equicontinuous at zero.
\item For each $\phi \in \Phi$, the sequence of distributions of $X^{n}(\phi)$ is uniformly tight on $D_{\infty}(\R)$.   
\item $\forall$ $m \in \N$, $\phi_{1}, \dots, \phi_{m} \in \Phi$, $t_{1}, \dots, t_{m} \geq 0$, the probability distribution of 
$ (X^{n}_{t_{1}}(\phi_{1}), \dots, X^{n}_{t_{m}}(\phi_{m}))$ converges in distribution to some probability measure on $\R^{m}$.   
\end{enumerate}
Then there exist a weaker countable Hilbertian topology $\theta$ on $\Phi$ and some $D_{\infty}((\widetilde{\Phi_{\theta}})'_{\beta})$-valued random variables $Y$ and $Y^{n}$, $ n \in \N$, such that  \begin{enumerate}[label=(\alph*)]
\item $\forall \, \phi \in \Phi$, $n \in \N$, the real-valued c\`{a}dl\`{a}g processes $X^{n}(\phi)$ and $Y^{n}[\phi]$ are indistinguishable, 
\item the sequence $( Y^{n}: n \in \N)$ is tight on $D_{\infty}((\widetilde{\Phi_{\theta}})'_{\beta})$, and
\item $Y^{n} \Rightarrow Y$ in $D_{\infty}((\widetilde{\Phi_{\theta}})'_{\beta})$.    
\end{enumerate}
Moreover $ (b)$ and $(c)$ are also satisfied for $Y$ and $(Y^{n}:n \in \N)$ as $D_{\infty}(\Phi'_{\beta})$-valued random variables.
\end{theo} 
\begin{prf}
First for each $n \in \N$, from \emph{(1)}, \emph{(2)} and Theorem \ref{theoRegulaTheoSkorokSpace}, there exists a $D_{\infty}(\Phi'_{\beta})$-valued random variable $Y^{n}$ such that for each $\phi \in \Phi$, $X^{n}(\phi)$ and $Y^{n}[\phi]$ are indistinguishable. 

For each $n \in \N$, let $\mu_{n}$ denote the probability distribution of $Y^{n}$ on $D_{\infty}(\Phi'_{\beta})$. Then, $\forall$ $m \in \N$, $\phi_{1}, \dots, \phi_{m} \in \Phi$, $t_{1}, \dots, t_{m} \geq 0$, it is clear that $\mu_{n} \circ (\Pi^{\phi_{1}, \dots, \phi_{m}}_{t_{1}, \dots, t_{m}})^{-1}$ is the probability distribution of  $ (X^{n}_{t_{1}}(\phi_{1}), \dots, X^{n}_{t_{m}}(\phi_{m}))$. In particular, for each $n \in \N$, the Fourier transform $\widehat{\mu}_{n,t}$ of $\mu_{n}$ at time $t$  is that of $X^{n}$ as a cylindrical random variable in $D_{\infty}(\Phi'_{\beta})$. 
Therefore, conditions \emph{(2)}, \emph{(3)} and \emph{(4)}, imply that the sequence $(\mu_{n}: n \in \N)$ satisfies the conditions in Theorem \ref{theoWeakConveMeasures}.  
This shows the existence of the a weaker countable Hilbertian topology $\theta$ on $\Phi$ and a  probability measure $\mu$ on $D_{\infty}((\widetilde{\Phi_{\theta}})'_{\beta})$ such that $\mu_{n} \Rightarrow \mu$ in $\goth{M}^{1}(D_{\infty}((\widetilde{\Phi_{\theta}})'_{\beta}))$. Hence each $Y^{n}$ is a $D_{\infty}((\widetilde{\Phi_{\theta}})'_{\beta})$-valued random variable and $Y$ is a $D_{\infty}((\widetilde{\Phi_{\theta}})'_{\beta})$-valued random variable whose probability distribution is $\mu$ (this is a consequence of L\'{e}vy's theorem and Corollary \ref{coroCanonicalRV}). Therefore $(a), (b), (c)$ are clearly satisfied. 
\end{prf}

In a similar way as for weak convergence of probability measures, under the assumption that the space $\Phi$ is ultrabornological and nuclear we can obtain a version of the above theorem with lower requirements on its assumptions. 

\begin{theo}  \label{theoWeakConvProcessesUltrab}
Let $\Phi$ be an ultrabornological nuclear space. For each $n \in \N$, let $X^{n}=\{X^{n}_{t} \}_{t \geq 0}$ be a $\Phi'_{\beta}$-valued c\`{a}dl\`{a}g process such that for each $t \geq 0$ the distribution of $X^{n}_{t}$ is a Radon measure on $\Phi'_{\beta}$. Suppose moreover that the sequence $(X^{n}: n \in \N)$ satisfies  (3) and (4) in Theorem \ref{theoWeakConvProcesses}. 
Then there exist a weaker countable Hilbertian topology $\theta$ on $\Phi$ and $(\widetilde{\Phi_{\theta}})'_{\beta}$-valued c\`{a}dl\`{a}g processes $Y= \{ Y_{t} \}_{t \geq 0}$ and $Y^{n}= \{ Y^{n}_{t} \}_{t \geq 0}$, $\forall \, n \in \N$, such that  
\begin{enumerate}[label=(\alph*)]
\item $\forall \, n \in \N$,  $X^{n}$ and $Y^{n}$ are indistinguishable, 
\item $Y$ and each $Y^{n}$ is a $D_{\infty}((\widetilde{\Phi_{\theta}})'_{\beta})$-valued random variable,
\item the sequence $( Y^{n}: n \in \N)$ is tight on $D_{\infty}((\widetilde{\Phi_{\theta}})'_{\beta})$, and
\item $Y^{n} \Rightarrow Y$ in $D_{\infty}((\widetilde{\Phi_{\theta}})'_{\beta})$.    
\end{enumerate}
Moreover, $(b), (c)$ are also satisfied for $Y$ and $(Y^{n}:n \in \N)$ as $D_{\infty}(\Phi'_{\beta})$-valued random variables. 
\end{theo} 
\begin{prf} 
The proof follows from very similar arguments to those used in Theorem \ref{theoWeakConvProcesses} from Corollary \ref{coroConditionRandVariSkoSpaceUltrab} and Theorem  \ref{theoWeakConveMeasuresUltrab}.  
\end{prf}

\section{Weak Convergence of L\'{e}vy Processes in the Skorokhod Space}\label{sectionWeakConvLevy}

\begin{assu}
For this section, $\Phi$ always denote a barrelled nuclear space.
\end{assu}

In this section we will provide sufficient conditions for a sequence of $\Phi'_{\beta}$-valued L\'{e}vy processes to converge in $D_{\infty}(\Phi'_{\beta})$. We start by recalling some basic properties of L\'{e}vy processes taking values in $\Phi'_{\beta}$. For further details see \cite{FonsecaMora:Levy}.  

 A $\Phi'_{\beta}$-valued process $L=\left\{ L_{t} \right\}_{t\geq 0}$ is called a \emph{L\'{e}vy process} if \begin{inparaenum}[(i)] \item  $L_{0}=0$ a.s., 
\item $L$ has \emph{independent increments}, i.e. for any $n \in \N$, $0 \leq t_{1}< t_{2} < \dots < t_{n} < \infty$ the $\Phi'_{\beta}$-valued random variables $L_{t_{1}},L_{t_{2}}-L_{t_{1}}, \dots, L_{t_{n}}-L_{t_{n-1}}$ are independent,  
\item L has \emph{stationary increments}, i.e. for any $0 \leq s \leq t$, $L_{t}-L_{s}$ and $L_{t-s}$ are identically distributed, and  
\item for every $t \geq 0$ the distribution $\mu_{t}$ of $L_{t}$ is a Radon measure and the mapping $t \mapsto \mu_{t}$ from $\R_{+}$ into the space $\goth{M}_{R}^{1}(\Phi'_{\beta})$ of Radon probability measures on $\Phi'_{\beta}$ is continuous at $0$ when $\goth{M}_{R}^{1}(\Phi'_{\beta})$  is equipped with the weak topology. \end{inparaenum}

Every $\Phi'_{\beta}$-valued L\'{e}vy process $L=\left\{ L_{t} \right\}_{t\geq 0}$ has a regular, c\`{a}dl\`{a}g version $\tilde{L}=\{ \tilde{L}_{t} \}_{t \geq 0}$ that is also a L\'{e}vy process. Moreover, there exists a weaker countably Hilbertian topology $\vartheta_{L}$ on $\Phi$ such that $\tilde{L}$ is a $(\widetilde{\Phi_{\vartheta_{L}}})'_{\beta}$-valued c\`{a}dl\`{a}g process (see \cite{FonsecaMora:Levy}, Corollary 3.11). Therefore, $L$ can be identified with a  $D_{\infty}(\Phi'_{\beta})$-valued random variable whose probability distribution is a Radon measure on $D_{\infty}(\Phi'_{\beta})$ (see the proof of Theorem \ref{theoRegulaTheoSkorokSpace}). 

Recall that a Borel measure $\nu$ on $\Phi'_{\beta}$ is a L\'{e}vy measure (see \cite{FonsecaMora:Levy}) if it satisfies: 
\begin{enumerate}
\item $\nu (\{ 0 \})=0$, 
\item for each neighborhood of zero $U \subseteq \Phi'_{\beta}$, the  restriction $\restr{\nu}{U^{c}}$ of $\nu$ on the set $U^{c}$ belongs to the space $\goth{M}^{b}_{R}(\Phi'_{\beta})$ of bounded Radon measures on $\Phi'_{\beta}$,    
\item there exists a continuous Hilbertian seminorm $\rho$ on $\Phi$ such that 
\begin{equation} \label{integrabilityPropertyLevyMeasure}
\int_{B_{\rho'}(1)} \rho'(f)^{2} \nu (df) < \infty,  \quad \mbox{and} \quad  \restr{\nu}{B_{\rho'}(1)^{c}} \in \goth{M}^{b}_{R}(\Phi'_{\beta}). 
\end{equation}
\end{enumerate}

One of the most important properties of a $\Phi'_{\beta}$-valued L\'{e}vy process $L=\left\{ L_{t} \right\}_{t\geq 0}$ is the L\'{e}vy-Khintchine formula for its Fourier transform  (\cite{FonsecaMora:Levy}, Theorem 4.18): for each $t \geq 0$, $\phi \in \Phi$, 
\begin{equation} \label{levyKhintchineFormulaLevyProcess}
\begin{split}
& \Exp \left( e^{i L_{t}[\phi] } \right) = e^{t \eta(\phi)}, \quad  \mbox{ with} \\ 
& \eta(\phi)= i \goth{m}[\phi] - \frac{1}{2} \mathcal{Q}(\phi)^{2} + \int_{\Phi'_{\beta}} \left( e^{i f[\phi]} -1 - i f[\phi] \ind{ B_{\rho'}(1)}{f} \right) \nu(d f).  
\end{split}
\end{equation}
where $\goth{m} \in \Phi'_{\beta}$,  $\mathcal{Q}$ is a continuous Hilbertian seminorm on $\Phi$, $\nu$ is a L\'{e}vy measure on $\Phi'_{\beta}$ and $\rho$ is a continuous Hilbertian seminorm  on $\Phi$ for which $\nu$ satisfies \eqref{integrabilityPropertyLevyMeasure}.
%Furthermore, because $\nu$  is a L\'{e}vy measure on $\Phi'_{\beta}$ it follows that $\nu$ is a $\sigma$-finite Radon measure (see \cite{FonsecaMora:Levy}, Proposition 4.9).

Our main result on convergence of L\'{e}vy processes is the following:

\begin{theo}\label{theoWeakConvLevyProcess}
For every $n \in \N$, let $L^{n}=\left\{ L^{n}_{t} \right\}_{t\geq 0}$ be a $\Phi'_{\beta}$-valued c\`{a}dl\`{a}g L\'{e}vy process where $(\goth{m}_{n}, \mathcal{Q}_{n}, \nu_{n}, \rho_{n})$ are as in \eqref{levyKhintchineFormulaLevyProcess}. Assume that there exists a continuous Hilbertian seminorm $q$ on $\Phi$ such that $\mathcal{Q}_{n} \leq q$ and $\rho_{n} \leq q$ $\forall n \in \N$, and such that the following is satisfied:
\begin{enumerate}
\item \label{condDrift} $(\goth{m}_{n}: n \in \N)$ is relatively compact in $\Phi'_{\beta}$,
\item \label{condGaussCova}  $\displaystyle{\sup_{n \in \N} \norm{i_{\mathcal{Q}_{n},q}}_{\mathcal{L}_{2}(\Phi_{q}, \Phi_{\mathcal{Q}_{n}})} < \infty}$, 
\item \label{condLevyMea} $\displaystyle{\sup_{n \in \N} \int_{\Phi'_{\beta}}(q'(f)^{2} \wedge 1) \nu_{n}(df)<\infty. }$
\end{enumerate}
Suppose moreover that  $\forall \, m \in \N$, $\phi_{1}, \dots, \phi_{m} \in \Phi$, $t_{1}, \dots, t_{m} \in [0,T]$, the sequence of distributions of 
$ (L^{n}_{t_{1}}(\phi_{1}), \dots, L^{n}_{t_{m}}(\phi_{m}))$ converges in distribution to some probability measure on $\R^{m}$. Then the sequence $( L^{n}: n \in \N)$ is uniformly tight on $D_{\infty}(\Phi'_{\beta})$ and there exists a $\Phi'_{\beta}$-valued L\'{e}vy process $L=\left\{ L_{t} \right\}_{t\geq 0}$ such that $L^{n} \Rightarrow L$ in $D_{\infty}(\Phi'_{\beta})$.
\end{theo}
\begin{prf}
For each $n\in \N$, let $\mu_{n}$ be the distribution of $L^{n}$ as a random variable in $D_{\infty}(\Phi'_{\beta})$. Then for each $t \geq 0$, the Fourier transform $\widehat{\mu}_{n,t}$ of $\mu_{n} \circ \Pi^{-1}_{t}$ is precisely the Fourier transform of $L^{n}_{t}$. Hence if for each $n \in \N$, $\eta_{n}$ is defined by \eqref{levyKhintchineFormulaLevyProcess},  then we have that 
$\widehat{\mu}_{n,t}(\phi)= e^{t \eta_{n}(\phi)}$ for each $t \geq 0$ and $\phi \in \Phi$. Then, in order to show that for every $T>0$ we have that $(\widehat{\mu}_{n,t}: t \in [0,T], n \in \N)$ is equicontinuous at zero, it is enough to show that $(\widehat{\mu}_{n,1}: n \in \N)$ is equicontinuous at zero. 

Now, because the family of measures $(\mu_{n} \circ \Pi^{-1}_{1}: n \in \N)$ is infinitely divisible (they correspond to the distributions of the sequence $(L^{n}_{1}: n \in \N)$; see \cite{FonsecaMora:Levy}, Theorem 3.5), then it follows from Satz 2.8 in \cite{Dettweiler:1976} that conditions \ref{condDrift}, \ref{condGaussCova}, \ref{condLevyMea} imply that the family  $(\mu_{n} \circ \Pi^{-1}_{1}: n \in \N)$ is uniformly tight on $\Phi'_{\beta}$. But as $\Phi$ is a barrelled space, then the above implies that $(\widehat{\mu}_{n,1}: n \in \N)$ is equicontinuous at zero (see \cite{DaleckyFomin}, Theorem III.2.7, p.104). Hence, $(\widehat{\mu}_{n,t}: t \in [0,T], n \in \N)$ is equicontinuous at zero for each $T>0$. 

Furthermore for every $t \geq 0$ and $\phi \in \Phi$ the sequence $(L^{n}_{t}[\phi])$ converges weakly and because $L^{n}[\phi]=(L^{n}_{t}[\phi])_{t \geq 0}$ is a real-valued c\`{a}dl\`{a}g L\'{e}vy process for each $n \in \N$, then for each $\phi \in \Phi$ the sequence $(L^{n}[\phi]: n \in \N)$ converges weakly in $D_{\infty}(\R)$ (see \cite{AndersenEtAl:2015}, Proposition 12.4). Therefore by Prokhorov theorem $(L^{n}[\phi]: n \in \N)$ is uniformly tight on $D_{\infty}(\R)$ and hence $(\mu_{n} \circ \Pi^{-1}_{\phi}: n \in \N)$ is uniformly tight on $D_{\infty}(\R)$. 

Then we have that all the conditions in Theorem \ref{theoWeakConvProcesses} are satisfied, and so we have that $( L^{n}: n \in \N)$ is uniformly tight on $D_{\infty}(\Phi'_{\beta})$ and also the existence of a $\Phi'_{\beta}$-valued c\`{a}dl\`{a}g process $L=\left\{ L_{t} \right\}_{t\geq 0}$ such that $L^{n} \Rightarrow L$ in $D_{\infty}(\Phi'_{\beta})$. Finally because $\forall \, m \in \N$, $\phi_{1}, \dots, \phi_{m} \in \Phi$, $t_{1}, \dots, t_{m} \in [0,T]$, $ (L^{n}_{t_{1}}(\phi_{1}), \dots, L^{n}_{t_{m}}(\phi_{m}))$ converges in distribution to 
$ (L_{t_{1}}(\phi_{1}), \dots, L_{t_{m}}(\phi_{m}))$, then it follows that $L$ is a cylindrical L\'{e}vy process. But then, $L$ is a $\Phi'_{\beta}$-valued L\'{e}vy process by Theorem 3.8 in \cite{FonsecaMora:Levy}. 
\end{prf}

\section{Tightness on the Skorokhod space of a Locally Convex Space}
\label{sectionApplSkoLocalConv}

% (definitions are analogue to those given in Sections \ref{sectionSkoSpac} and \ref{sectionUNSPINFTY})
\begin{assu}
For this section, $(\Phi, \tau)$ always denote a (Hausdorff) locally convex space.
\end{assu}

In this section we will show how the machinery developed on the last sections for the case of the dual of a nuclear space can be applied to study uniform tightness of probability  measures on $D_{\infty}(\Phi'_{\beta})$. This can be done through the use of the Sazonov topology whose definition will be recalled for the convenience of the reader. 
For further details the reader is referred to \cite{BogachevMT, SchwartzRM, SmolyanovFomin}. 

Let $\mathscr{P}(\Phi, \tau)$ denote the collection of all the seminorms on $(\Phi, \tau)$ defined in the following way: $p \in  \mathscr{P}(\Phi, \tau)$ if and only if $p$ is a continuous Hilbertian seminorm on $\Phi$ for which there exists a continuous Hilbertian seminorm $q$ on $\Phi$ such that $p \leq q$, $\Phi_{q}$ is separable, and such that the canonical inclusion $i_{p,q}: \Phi_{q} \rightarrow \Phi_{p}$ is Hilbert-Schmidt. The collection $\mathscr{P}(\Phi, \tau)$ is not-empty as every seminorm on $\Phi$ continuous with respect to the weak topology $\sigma$ is a member of $\mathscr{P}(\Phi, \tau)$.  

The locally convex topology on $\Phi$ generated by the family of seminorms $\mathscr{P}(\Phi, \tau)$ is called the \emph{Sazonov topology} on $\Phi$ with respect to the topology $\tau$ and is denote by $\tau_{S}$. Considering finite dimensional subspaces of $\Phi$ as Hilbert spaces, it is clear  that $\sigma$ is weaker than $\tau_{S}$. On the other hand, each $p \in  \mathscr{P}(\Phi, \tau)$ is a continuous Hilbertian seminorm on $\Phi$ and therefore we have that $\tau_{S}$ is weaker than $\tau$. Moreover, $\tau_{S} = \tau$ if and only if $(\Phi, \tau)$ is a nuclear space.

\begin{theo}\label{theoUniformTightnessInftylocalConvex} 
Let $(\mu_{\alpha}: \alpha \in A)$ be a family of probability measures on $D_{\infty}(\Phi'_{\beta})$ that satisfies the following conditions:
\begin{enumerate}
\item  For all $T>0$, the family of Fourier transforms $(\widehat{\mu}_{\alpha,t}: t \in [0,T], \alpha \in A)$ is equicontinuous at zero on $(\Phi,\tau_{S})$.
\item  For each  $\phi \in \Phi$, the family $(\mu_{\alpha} \circ \Pi_{\phi}^{-1}: \alpha \in A)$ of probability measures on $D_{\infty}(\R)$ is uniformly tight.    
\end{enumerate}  
Then there exists a weaker countably Hilbertian topology $\theta$ on $\Phi$ such that $(\mu_{\alpha}: \alpha \in A)$ is uniformly tight on $D_{\infty}((\widetilde{\Phi_{\theta}})'_{\beta})$. In particular, the family $(\mu_{\alpha}: \alpha \in A)$ is uniformly tight on $D_{\infty}(\Phi'_{\beta})$. 
\end{theo}
\begin{prf}
The proof of this theorem can be done from a modification of the arguments used in the proofs of Theorems \ref{theoThighnessMeasures} and  \ref{theoUniformTightnessInfty}. For the benefit of the reader we will sketch the main steps. First, the regularization theorem (Theorem \ref{theoRegularizationTheoremCadlagContinuousVersion}) remains valid if we assume equicontinuity with respect to the Sazonov topology $\tau_{S}$ (see \cite{FonsecaMora:ReguCyliLCS}). Then the regularization theorem and the Minlos theorem on Skorokhod space (Theorems \ref{theoRegulaTheoSkorokSpace} and \ref{theoMinlosSkorokhodSpace}) (and therefore Proposition \ref{propFamilyMeasRadonInCHS}) remain valid if we assume equicontinuity with respect to $\tau_{S}$. In a similar way, Theorem \ref{theoThighnessMeasures} can be proved if we assume equicontinuity of the Fourier transforms with respect to the Sazonov topology $\tau_{S}$. In effect, it follows from our assumptions that in \eqref{inequaSupCharaFunct} we can choose the continuous Hilbertian seminorm $p$ to be $\tau_{S}$-continuous. Therefore, from the definition of the topology $\tau_{S}$ and since every Hilbert-Schmidt operator can be factored into the composition of a Hilbert-Schmidt operator and a compact operator (see \cite{SchwartzRM}, Proposition II.3.6, p.217), for the  $\tau_{S}$-continuous seminorm $p$ one can find two continuous Hilbertian seminorms $q$ and $\varrho$ on $\Phi$ such that the canonical inclusions $i_{p,q}: \Phi_{q} \rightarrow \Phi_{p}$ and $i_{q,\varrho}: \Phi_{\varrho} \rightarrow \Phi_{q}$ are respectively Hilbert-Schmidt and compact. The proof of Theorem \ref{theoThighnessMeasures} can now be replicated  with almost no changes. Finally, in the proof of Theorem \ref{theoUniformTightnessInfty} we only used the corresponding result from Theorem \ref{theoThighnessMeasures}, but since the conclusions of the latter theorem remains valid with the assumption of equicontinuity with respect to the Sazonov topology $\tau_{S}$, the conclusions of Theorem \ref{theoUniformTightnessInfty} are valid as well. In the above result we conclude tightness  of the family $(\mu_{\alpha}: \alpha \in A)$ on $D_{\infty}((\Phi,\tau_{S})'_{\beta})$, but since 
the inclusion from $D_{\infty}((\Phi,\tau_{S})'_{\beta})$ into 
$D_{\infty}(\Phi'_{\beta})$ is continuous (this as the topology  on $(\Phi,\tau_{S})'_{\beta}$ is finer that the induced topology from $\Phi'_{\beta}$) we conclude tightness of $(\mu_{\alpha}: \alpha \in A)$ on $D_{\infty}(\Phi'_{\beta})$. 
\end{prf}

In a similar way, modifiying the arguments in the proof of Theorem \ref{theoWeakConveMeasures} by using Theorem \ref{theoUniformTightnessInftylocalConvex} we can prove the following: 

\begin{theo} \label{theoWeakConveMeasureslocallyConvex}
Let $(\mu_{n}: n \in \N)$ be a sequence of probability measures on $D_{\infty}(\Phi'_{\beta})$ that satisfies \emph{(1)} and \emph{(2)} of Theorem \ref{theoUniformTightnessInftylocalConvex}, and such that for all $m \in \N$, $\phi_{1}, \dots, \phi_{m} \in \Phi$, $t_{1}, \dots, t_{m} \geq 0$, the sequence $ \mu_{n} \circ \left( \Pi^{\phi_{1}, \dots, \phi_{m}}_{t_{1}, \dots, t_{m}} \right)^{-1}$ converges weakly on $\R^{m}$.
Then, $\mu_{n} \Rightarrow \mu$ in $\goth{M}^{1}(D_{\infty}(\Phi'_{\beta}))$. 
\end{theo}

\begin{rema}
It should be clear to the reader that we can also prove a version of Theorem \ref{theoWeakConvProcesses} in the  context of locally convex spaces provided that we assume that the linear maps $\{ X^{n}_{t}: t \in [0,T], n \in \N \}$ from $\Phi$ into $L^{0} \ProbSpace$ are equicontinuous at zero on $(\Phi,\tau_{S})$. We leave the details to the reader.  
\end{rema}

% A generalized version of the Minlos theorem (see \cite{SmolyanovFomin}, Theorem IV.1) says that if $M$ is a set of cylindrical measures on $\Phi'$, and if the set $\widetilde{M}$ of its Fourier transforms is continuous at zero on $(\Phi,\tau_{S})$, then $\mu$ has a Radon extension on $\Phi'_{\beta}$.   In this context, $\tau_{S}= \tau$ if and only if $H$ is finite dimensional.

If $H$ is a separable Hilbert space, recall that the Sazonov topology $\tau_{S}$ on $H$ is generated by the seminorms on $H$ of the form $p_{S}(\phi)=\norm{S\phi}_{H}$ $\forall \, \phi \in H$, where $S$ runs over the totally of all Hilbert-Schmidt operators on $H$. This definition of Sazonov topology on $H$ is equivalent to the topology introduced above for general locally convex spaces (see \cite{SmolyanovFomin}). Since the conclusions of Theorem  \ref{theoMinlosSkorokhodSpace} remain valid if we assume equicontinuity with respect to the Sazonov topology, we obtain the following generalization of Sazonov's theorem on the Skorokhod space:  

\begin{theo} \label{theoSazonovSkorokhodHilbertSpace}
Let $H$ be a separable Hilbert space and let $\mu$ be a cylindrical probability measure on $D_{T}(H)$. Suppose that the family of its Fourier transforms $(\widehat{\mu}_{t}: t \in [0,T])$ is equicontinuous at zero on $(H, \tau_{S})$. Then there exists a Radon probability measure $\nu$ on $D_{T}(H)$ that is an extension of $\mu$. 
\end{theo}

Now, if we use Theorems \ref{theoUniformTightnessInftylocalConvex} and \ref{theoWeakConveMeasureslocallyConvex} we obtain the following generalization of L\'{e}vy's continuity theorem on the Skorokhod space: 
 
\begin{theo} \label{theoLevyTheoremSkorokhodHilbertSpace} 
Let $H$ be a separable Hilbert space. For a sequence $(\mu_{n}: n \in \N)$ of probability measures on $D_{\infty}(H)$ to be uniformly tight it is sufficient that the following conditions are satisfied:
\begin{enumerate}
\item  For all $T>0$, the family of Fourier transforms $(\widehat{\mu}_{n,t}: t \in [0,T], n \in \N)$ is equicontinuous at zero on $(H, \tau_{S})$.
\item  For each  $h \in H$, the sequence $(\mu_{n} \circ \Pi_{h}^{-1}: n \in \N)$ of probability measures on $D_{\infty}(\R)$ is uniformly tight.    
\end{enumerate}  
If moreover $\forall \, m \in \N$, $h_{1}, \dots, h_{m} \in H$, $t_{1}, \dots, t_{m} \geq 0$, the sequence $ \mu_{n} \circ \left( \Pi^{h_{1}, \dots, h_{m}}_{t_{1}, \dots, t_{m}} \right)^{-1}$ converges weakly on $\R^{m}$,
 then $\mu_{n} \Rightarrow \mu$ in $\goth{M}^{1}(D_{\infty}(H))$.  
\end{theo}

We hope that Theorems \ref{theoSazonovSkorokhodHilbertSpace} and \ref{theoLevyTheoremSkorokhodHilbertSpace} can serve as  useful sufficient conditions for tightness and weak convergence on the Skorokhod space of a Hilbert space; this as a contribution to the literature on the Skorokhod space of a metric space (see e.g. \cite{Billingsley, EthierKurtz, Skorokhod:1956}). 

\section{Remarks and Examples} \label{sectionExampCommen}

All throughout this paper we have considered random variables and probability measures on the dual of a nuclear space. Most of our results have been formulated under the context of a general nuclear space $\Phi$, but in certain occasions we have assumed some additional structure on $\Phi$, for example that $\Phi$ is barrelled or ultrabornological. The purpose of  this section is to provide concrete examples of nuclear spaces satisfying these conditions and to attribute to each of them the properties used throughout the paper and thus the results valid for them. Some additional remarks are given and comparison of our results with those obtained by other authors.

\subsection{The case of ultrabornological and barrelled nuclear spaces.}

\textbf{Examples:} There are many examples of spaces of functions widely used in analysis  that are nuclear spaces. For example, it is known (see e.g. \cite{Pietsch, Schaefer, Treves}) that the spaces functions $\mathscr{E}_{K} \defeq \mathcal{C}^{\infty}(K)$ ($K$: compact subset of $\R^{d}$), $\mathscr{E}\defeq \mathcal{C}^{\infty}(\R^{d})$, the rapidly decreasing functions $\mathscr{S}(\R^{d})$, and the space of harmonic functions $\mathcal{H}(U)$ ($U$: open subset of $\R^{d}$),  are all  examples of Fr\'{e}chet nuclear spaces. Their (strong) dual spaces $\mathscr{E}'_{K}$, $\mathscr{E}'$, $\mathscr{S}'(\R^{d})$, $\mathcal{H}'(U)$, are also nuclear spaces.
On the other hand, the space of test functions $\mathscr{D}(U) \defeq \mathcal{C}_{c}^{\infty}(U)$ ($U$: open subset of $\R^{d}$), the space of polynomials $\mathcal{P}_{n}$ in $n$-variables, the space of real-valued sequences $\R^{\N}$ (with direct sum topology) are strict inductive limits of Fr\'{e}chet nuclear spaces (hence they are also nuclear). The space of distributions  $\mathscr{D}'(U)$  ($U$: open subset of $\R^{d}$) is also nuclear. All the above are examples of (complete) ultrabornological nuclear spaces.   

\textbf{Compactness on Skorokhod space:} In Theorem \ref{theoCharacCompacSets} we have proved  that for a barrelled nuclear space $\Phi$ and $A \subseteq D_{T}(\Phi'_{\beta})$, compactness of finite dimensional projections of $A$ implies compactness. The same  characterization  but for the case when $\Phi$ is Fr\'{e}chet nuclear space was proved by Mitoma in \cite{Mitoma:1983}. We are not aware of any further extension of the result of Mitoma to any other classes of nuclear spaces. Moreover, since any ultrabornological space is also barrelled, the examples given in the previous paragraph are all examples of barrelled nuclear spaces for which Theorem \ref{theoCharacCompacSets} is valid. 

Now, it is important to mention that there are examples of (complete) nuclear spaces that are not barrelled and for which compactness of its finite-dimensional projections do not implies compactness (and hence the characterization given in Theorem \ref{theoCharacCompacSets} fails). To provide an example, let $E$ be an infinite dimensional Banach space. Then, it is known that $E$ is the strong dual of some complete nuclear space $\Phi$ (see Corollary 1 of Theorem IV.4.3.3 in  \cite{HogbeNlendMoscatelli}). Note that $\Phi$ cannot be barrelled because if that were the case the space $\Phi$ would be reflexible and  then $\Phi=E'_{\beta}$ (see \cite{Schaefer}, Theorems III.7.2 and IV.5.6). But this equality is impossible because in that case $\Phi$ would be both nuclear and Banach, and this is only possible if $\Phi$ is finite dimensional (see \cite{Pietsch}, Theorem 4.4.14) so we get a contradiction. 
Now, let $B$ denote the closed unit ball in $E$ and $\delta >0$. Let $A=A(B, \delta)$ denote the collection of all $x \in D_{T}(E)=D_{T}(\Phi'_{\beta})$ that are of the form $x(t) = f_{j}$ for $t \in [t_{j-i},t_{j})$, $j = 1,\dots,m$, where $t_{j}-t_{j-1}> \delta$, $f_{j} \in B$, $t_{0}=0$, $t_{m}=T$.
Observe that for each $\phi \in \Phi$, the set $\Pi_{\phi}(A)=\{ x[\phi]: x \in A\}$ is relatively compact in $D_{T}(\R)$. This is a consequence of the fact that $B[\phi]=\{ f[\phi]: f \in B\}$ is relatively compact in $\R$ and Lemma 2.4.1 in \cite{KallianpurXiong}. However, the set $A$ cannot be relatively compact in $D_{T}(E)$ because in that case the closure of the set $\{ x(t): t \in [0,T], x \in A\}$, that is equal to $B$, must be compact (see the proof of Proposition 1.6.vi) in \cite{Jakubowski:1986}); but this is imposible as $E$ is infinite dimensional. 

\textbf{Tightness and weak convergence on Skorokhod space:} In Theorems \ref{theoTightMeasuresUltrab} and \ref{theoUniformTightnessUltrabInfty} we have proved that if $\Phi$ is an ultrabornological space, tightness of a family of probability measures on the Skorokhod space in $\Phi'_{\beta}$ is equivalent to tightness of its one-dimensional projections on the Skorokhod space in $\R$. Under the same hypothesis on $\Phi$, in Theorem \ref{theoWeakConveMeasuresUltrab} we proved that for weak convergence of a sequence of probability measures on the Skorokhod space in $\Phi'_{\beta}$ it is sufficient to have tightness of one-dimensional projections on $D_{\infty}(\R)$ and weak convergence of time-space finite-dimensional projections. The analogous result for weak convergence of a sequence of $\Phi'_{\beta}$-valued c\`{a}dl\`{a}g processes is given in Theorem \ref{theoWeakConvProcessesUltrab}. 

The above results were firstly proved by Mitoma in \cite{Mitoma:1983} for the case when $\Phi$ is a Fr\'{e}chet nuclear space. These results were later extended by Fouque in \cite{Fouque:1984} to the case when $\Phi$ is a countable inductive limit of nuclear  Fr\'{e}chet spaces.
However, if the space $\Phi$ is a Fr\'{e}chet space or the countable inductive limit of Fr\'{e}chet spaces, then $\Phi$ is an ultrabornological space (see \cite{Jarchow}, Corollaries 4 and 5, Section 13.1, p.273). Hence, Theorems \ref{theoTightMeasuresUltrab}, \ref{theoUniformTightnessUltrabInfty},  \ref{theoWeakConveMeasuresUltrab} and \ref{theoWeakConvProcessesUltrab} generalize the results obtained by Mitoma and Fouque under the same hypothesis. Moreover, our results work for classes of ultrabornological nuclear spaces that are not covered by Mitoma and Fouque's assumptions, for example the space of real-analytic functions $\mathcal{A}(V)$ ($V$: closed subset of $\R^{d}$, see \cite{HogbeNlendMoscatelli}), or the space $\R^{\aleph}$ equipped with the product topology (where $\aleph$ denote the cardinal of the continuum, see \cite{Kalton:1970}).

\subsection{The case of general nuclear spaces.}

\textbf{Examples:} There are interesting examples of nuclear spaces that are not (or might not be) ultrabornological or barrelled. As examples we have the space $\R^{D}$ equipped with its product topology ($D$ arbitrary set, see \cite{Treves}) and the nuclear K\"{o}the sequence spaces (see \cite{HogbeNlendMoscatelli, Jarchow}). 

%Also, in \cite{Berner:1977} is shown that there exists a topology on the space of holomorphic functions $\mathcal{H}(V)$ defined on a domain $V$ in the space of distributions $\mathscr{D}'(\R^{d})$, and such that this topology is nuclear but not barrelled (hence not ultrabornological).

Many more examples can be generated if one consider spaces of functions defined on nuclear spaces or spaces whose strong dual is nuclear (also called dual nuclear spaces). One have for example the space of holomorphic functions defined on a  (quasi-)complete dual nuclear space (e.g. $\mathcal{H}(\mathscr{D}'(\R^{d}))$, see \cite{Dineen}), the space of continuous linear operators between a semi-reflexive dual nuclear space into a nuclear space (e.g. the space $\mathscr{D}'(U; \R^{\N}) \defeq \mathcal{L}(\mathcal{C}^{\infty}_{c}(U),\R^{\N})$ of distributions with values in $\R^{\N}$, $U \subseteq \R^{n}$ open), and tensor products of arbitrary nuclear spaces (e.g. the space of holomorphic functions with values in the space of distributions $\mathcal{H}(U;\mathscr{D}'(\R^{d})) \cong \mathcal{H}(U) \widehat{\otimes} \mathscr{D}'(\R^{d})$, $U \subseteq \R^{n}$ open); for references see \cite{Schaefer, Treves}. A particular example of a not-barrelled nuclear space that has been used on the study of 
weak convergence of a sequence of $\mathscr{S}'(\R^{d})$-valued c\`{a}dl\`{a}g processes is the space 
$\mathscr{D}(\R) \widehat{\otimes} \mathscr{S}(\R^{d}) $ (see \cite{BojdeckiGorostizaRamaswamy:1986} for details). 

\textbf{Tightness and weak convergence on Skorokhod space:} 
Apart from studying tightness and weak convergence of probability measures and random variables in the Skorokhod space $D_{T}(\Phi'_{\beta})$, in this article we have also introduced the more general concepts of cylindrical measures and cylindrical random variables in $D_{T}(\Phi'_{\beta})$. We are not aware of any other work that makes a systematic study of these objects. In particular, for the case of $\Phi$ being a nuclear space we have shown that there is an analogue of the regularization theorem (Theorem \ref{theoRegulaTheoSkorokSpace}) and of Minlos theorem (Theorem \ref{theoMinlosSkorokhodSpace}) for cylindrical random variables and cylindrical measures in $D_{T}(\Phi'_{\beta})$.  

Now, if $\Phi$ is a general nuclear space, in Theorems \ref{theoThighnessMeasures} and  \ref{theoUniformTightnessInfty} we have shown that for a family of probability measures on $D_{T}(\Phi'_{\beta})$ (or on $D_{\infty}(\Phi'_{\beta})$), the equicontinuity of Fourier transforms and  tightness of its one-dimensional projections on the Skorokhod space in $\R$ is sufficient for tightness on $D_{T}(\Phi'_{\beta})$ (or on $D_{\infty}(\Phi'_{\beta})$), and necessary if $\Phi$ is barrelled. There are two important comments we want to make on this result. First, the methodology used in our proofs shows that the use of the \emph{Baire category argument}, that was a common factor on the arguments of Mitoma and Fouque (see the proof of Lemma 3.3 in \cite{Mitoma:1983}, and of Lemme II.1 and Lemme IV.1 in \cite{Fouque:1984}), can be replaced in a very efficient way by the use of a weaker countably Hilbertian topology. This was a fundamental step in our arguments to extend the results of Mitoma to general nuclear spaces. Observe that equicontinuity of Fourier transforms played a fundamental role to show this existence of the weaker countably Hilbertian topology. 

Second, our condition of equicontinuity of Fourier transforms seems to be a less demanding condition than the \emph{compact containment condition} introduced by Jakubowski in \cite{Jakubowski:1986}. This condition, together with 
weak tightness with respect to some separating family, constitutes the main characterization for tightness on the Skorokhod space for a completely regular space whose compact subsets are metrizable. The more recent work of Kouritzin \cite{Kouritzin:2016} replaces the weak tightness assumption for some modulus of continuity conditions, but the compact containment condition is still present. Apart from the argument given above, observe that unlike the results in \cite{Jakubowski:1986} and \cite{Kouritzin:2016}, in Theorems \ref{theoThighnessMeasures} and \ref{theoUniformTightnessInftylocalConvex} we we do not need to assume that the compact subsets of $\Phi'_{\beta}$ are metrizable. The importance of this remark is in the fact that not for every locally convex space its compact subsets are metrizable. Necessary and sufficient conditions for the above to be true are given in \cite{FerrandoEtAl:2006}. 

In Theorem \ref{theoWeakConveMeasures} we introduced an analogue of Levy's continuity theorem for weak convergence of probability measures on $D_{\infty}(\Phi'_{\beta})$. The corresponding result for weak convergence of (cylindrical) processes in $\Phi'_{\beta}$ is given in Theorem  \ref{theoWeakConvProcesses}. We have illustrated the usefulness of these results in Theorem \ref{theoWeakConvLevyProcess} where we provide sufficient conditions for the weak convergence of a sequence of L\'{e}vy processes in $D_{\infty}(\Phi'_{\beta})$. We hope we could apply our results to the study of weak convergence of SPDEs taking values in the dual of a nuclear space defined by the author in \cite{FonsecaMora:2018-1}; this is done for example in \cite{FernandezGorostiza:1992, PerezAbreuTudor:1992} for  SPDEs in the dual of a Fr\'{e}chet nuclear space. 

\subsection*{Acknowledgements}
The author thank the referees for their helpful and insightful comments that contributed greatly  to improve the presentation of this manuscript.  
Thanks also to the The University of Costa Rica for providing financial support through the grant 820-B6-202	``Ecuaciones diferenciales en derivadas parciales en espacios de dimensi\'{o}n infinita''.

\end{document}